\documentclass[a4paper, 12pt]{amsart}

\usepackage[leqno]{amsmath}
\usepackage{amssymb,amsfonts,xfrac,MnSymbol}
\usepackage{amsthm,amsrefs}
\usepackage{hyperref}
\usepackage{todonotes}
\usepackage{enumitem}
\usepackage{graphicx}
\usepackage{tikz-cd, tikz}
\usepackage{color}
\usepackage{import}
\usepackage{cleveref}

\numberwithin{equation}{section}

\newtheorem{theorem}{Theorem}[section]
\newtheorem{claim}[theorem]{Claim}

\newtheorem{lemma}[theorem]{Lemma}
\newtheorem{corollary}[theorem]{Corollary}

\newtheorem{step}{Step}

\newtheorem*{theorem*}{Theorem}
\newtheorem*{claim*}{Claim}
\newtheorem*{proposition*}{Proposition}
\newtheorem*{lemma*}{Lemma}
\newtheorem*{corollary*}{Corollary}

\makeatletter
\newcommand{\newreptheorem}[2]{%
\newtheorem*{rep@#1}{\rep@title}%
\newenvironment{rep#1}[1]{%
 \def\rep@title{#2 \ref*{##1}}%
\begin{rep@#1}}%
{\end{rep@#1}}}
\makeatother
\newtheorem{theoremA}{Theorem}

\newreptheorem{theoremA}{Theorem}

\theoremstyle{definition}
\newtheorem{definition}[theorem]{Definition}
\newtheorem{observation}[theorem]{Observation}
\newtheorem{remark}[theorem]{Remark}
\newtheorem{example}[theorem]{Example}

\newtheorem{notation}[theorem]{Notation}
\newtheorem{setup}[theorem]{Setup}

\newtheorem{convention}[theorem]{Convention}

\newtheorem*{definition*}{Definition}
\newtheorem*{observation*}{Observation}
\newtheorem*{remark*}{Remark}
\newtheorem*{example*}{Example}
\newtheorem*{question*}{Question}
\newtheorem*{exercise*}{Exercise}
\newtheorem*{fact*}{Fact}
\newtheorem*{notation*}{Notation}

\newenvironment{claimproof}[1][\proofname]
  {%
    \proof[#1]%
  }
  {%
    \endproof%
  }

\newcommand{\bbA}{\mathbb{A}}

\newcommand{\bbC}{\mathbb{C}}

\newcommand{\bbG}{\mathbb{G}}

\newcommand{\bbN}{\mathbb{N}}
\newcommand{\bbO}{\mathbb{O}}
\newcommand{\bbP}{\mathbb{P}}

\newcommand{\bbZ}{\mathbb{Z}}

\newcommand{\bfA}{\mathbf{A}}
\newcommand{\bfB}{\mathbf{B}}
\newcommand{\bfC}{\mathbf{C}}

\newcommand{\bfG}{\mathbf{G}}

\newcommand{\bfK}{\mathbf{K}}
\newcommand{\bfL}{\mathbf{L}}

\newcommand{\frC}{\mathfrak{C}}

\newcommand{\calA}{\mathcal{A}}
\newcommand{\calB}{\mathcal{B}}
\newcommand{\calC}{\mathcal{C}}

\newcommand{\calE}{\mathcal{E}}

\newcommand{\calO}{\mathcal{O}}
\newcommand{\calP}{\mathcal{P}}

\newcommand{\actson}{\curvearrowright}

\newcommand{\inj}{\hookrightarrow}

\newcommand{\ii}{^{-1}}

\newcommand{\gen}[1]{\left< #1 \right>}

\newcommand{\tild}[1]{\widetilde{#1}}

\DeclareMathOperator{\id}{id}

\DeclareMathOperator{\Cay}{Cay}

\DeclareMathOperator{\rank}{rank}

\DeclareMathOperator{\im}{Im}
\DeclareMathOperator{\len}{len}
\DeclareMathOperator{\lab}{lab}

\definecolor{zachcomment}{rgb}{0.55,0.71,0}

\definecolor{nircomment}{rgb}{0.82, 0.1, 0.26}

\title[Ascending Chains in 3-Manifold and Rel. Hyperbolic Groups] {Ascending chains in 3-manifold and (relatively) hyperbolic groups}

\author[E. Bering]{Edgar A. Bering IV} 
\author[J. Heikamp]{Jakob Heikamp} 
\author[J. Kohav]{Jack Kohav}
\author[N. Lazarovich]{Nir Lazarovich} 
\author[Z. Munro]{Zachary Munro}

\thanks{Bering was partially supported by the Azreili Foundation and supported by the National Science Foundation under
Award No. DMS--2532818. Lazarovich was partially supported by the Israel Science 
  Foundation (grant no. 1576/23).
  Munro was supported by the Israel Science 
  Foundation (grant no. 1576/23) and the National Science Foundation under Award No. DMS--2503331.}

\begin{document}

\begin{abstract}
We prove that any ascending chain of bounded rank subgroups in the fundamental group of a compact $3$-manifold stabilizes. 
We use geometrization to reduce the proof to fundamental groups of complete, finite-volume hyperbolic $3$-manifolds. 
To handle this case, we prove the following: 
In a toral relatively hyperbolic group, any ascending chain of bounded rank, locally relatively quasiconvex subgroups stabilizes. 
We note this theorem is new even for bounded rank, locally quasiconvex chains in hyperbolic groups. 
\end{abstract}
\maketitle

 \section{Introduction}

A group is \emph{Noetherian} (or satisfies \emph{ACC})  if every ascending chain of subgroups stabilizes. 
Most groups of interest in geometric group theory fail to be Noetherian because they contain nonabelian free subgroups. 
However, it is still possible to achieve ascending chain conditions for restricted classes of subgroups.

\begin{definition}
    A group satisfies \emph{$\omega$ACC} if any ascending chain of uniformly bounded rank subgroups stabilizes.
\end{definition}

Takahashi~\cite{takahasi} and Higman~\cite{higman} independently proved that free groups satisfy $\omega$ACC. Kapovich and Myasnikov~\cite{kapovich-myasnikov} gave a geometric proof of this result using Stallings folds. 
Shusterman~\cite{shusterman} proved that limit groups satisfy $\omega$ACC (this includes fundamental groups of compact surfaces). We prove the following.

\begin{theoremA}
\label{main theorem 3-manifolds}
    Fundamental groups of compact 3-manifolds satisfy $\omega$ACC.
\end{theoremA}

\begin{theoremA}
\label{main theorem rel Hyperbolic}
    In a toral relatively hyperbolic group, any ascending chain of bounded rank, locally relatively quasiconvex subgroups stabilizes. 
\end{theoremA}

We note that this theorem is new even for hyperbolic groups, which are special case of toral relatively hyperbolic. 

The local quasiconvexity assumption in \Cref{main theorem rel Hyperbolic} is necessary. For example, in any ascending HNN-extension $$G=\gen{ F, t \mid tgt^{-1}=\varphi(g),\ g\in F },$$
where $\varphi:F\inj F$ has a proper image and $F$ is finitely generated, the bounded rank ascending chain $F< t^{-1}Ft< t^{-2}Ft^2<\dots$ does not stabilize. And there are many hyperbolic groups of this form.

To prove \Cref{main theorem 3-manifolds}, we begin by proving $\omega$ACC for the geometric pieces of the $3$-manifold. 
Complete finite-volume hyperbolic pieces are handled using \Cref{main theorem rel Hyperbolic}, while all other geometries are either Seifert-fibered spaces or Sol geometry and are handled in \Cref{no ascending in SFS}.
Finally we use the geometrization splitting and the following combination theorem to deduce that the fundamental group of a 3-manifold satisfies $\omega$ACC:

\begin{theoremA}
\label{main combination thm}
    Let $G$ be the fundamental group of a $k$-acylindrical, finite graph of groups. Suppose:
    \begin{itemize}
        \item The edge groups are finitely generated abelian groups.
        \item The vertex groups satisfy $\omega$ACC.
        \item $G$ is coherent.
        \item Torsion elements of $G$ have uniformly bounded order.
    \end{itemize} 
    Then $G$ satisfies $\omega$ACC.
\end{theoremA}

\subsection*{Acknowledgments.} 
The authors would like to thank Richard Weidmann for several insightful conversations and his patience in explaining the results of \cite{WeidWell}.








\section{Preliminaries}
\label{sec: preliminaries}
\subsection{Graph of groups}
We recall definitions regarding graphs of groups. For a more thorough introduction, the reader can consult Serre~\cite{serre2002trees}.

\begin{definition}[Graphs]
    A \emph{graph} $\bfA$ consists of a set of \emph{vertices} $V\bfA$, a set of \emph{edges} $E\bfA$, maps $e\mapsto e_+$ and $e\mapsto e_-$ from $E\bfA$ to $V \bfA$, and a fixed-point-free involution $e\mapsto e\ii $ of $E\bfA$ such that $(e\ii)_+ = e_-$ for all $e\in E\bfA$. A pair $\{e, e^{-1}\}$ is a \emph{geometric edge}, and we denote the set of geometric edges by $\mathcal E\mathbf A$.
    
    The \emph{degree} of $v\in V\bfA$ is $\deg(v) = |\{e\in E\bfA \;|\;e_-=v\}|$.
\end{definition}

\begin{definition}[Paths]
    A \emph{path} $p$ in $\bfA$ is a sequence $(v_0,e_1,v_1, \dots,e_n,v_n)$ such that 
    
    \begin{itemize}
        \item $v_j\in V\bfA$ and $e_i\in E\bfA$ for each $j=0,\dots, n$ and $i=1,\dots n$, and
        \item $(e_i)_-=v_{i-1}$ and $(e_i)_+=v_i$ for each $i=1,\dots, n$.
    \end{itemize}

    The \emph{endpoints} of $p$ are the vertices $v_0$, $v_n$ and are denoted $p_-$, $p_+$. A path is determined by its edges, so we will sometimes specify a path by a sequence $(e_1,\dots, e_n)$ such that $(e_i)_+=(e_{i+1})_-$ for $i=1,\dots, n-1$. The \emph{length} $\len(p)$ of $p$ is the number of edges $n$. 
\end{definition}

\begin{definition}[Graph of groups]
    A \emph{graph of groups} $\bbA$ consists of an \emph{underlying graph} $\bfA$, \emph{vertex groups} $A_v$ for $v\in V\bfA$, \emph{edge groups} $A_e$ for $e\in E\bfA$ such that $A_e = A_{e\ii}$, and monomorphisms $A_e \to A_{e_+}$ whose image we denote by $(A_e)_+$. Similarly, we denote by $(A_e)_-$ the image of $A_{e \ii}\to A_{e_-}$.
\end{definition}

\begin{definition}[Orientation]
    Let $\bfA$ be a graph. An \emph{orientation} of $\bfA$ is a choice of edge from each pair $\{e,e \ii\}$ for each $e\in E\bfA$. The chosen edge from each pair is \emph{oriented}. A graph of groups is \emph{oriented} if its underlying graph has an orientation.
\end{definition}

\begin{definition}[$\bbA$-path]
    Let $(v_0,e_1,v_1,\dots,e_n,v_n)$ be a path in $\bf A$. An \emph{$\bbA$-path} with \emph{underlying path} $(v_0,e_1,v_1,\dots, e_n,v_n)$ is a sequence $(a_0,e_1,a_1,\dots,e_n,a_n)$ such that $a_i\in A_{v_i}$ for $i=0,\dots, n$.
\end{definition}

\begin{definition}[Fundamental group $\pi_1(\bbA, v_0)$]
    Let $\sim$ be the equivalence relation on $\bbA$-paths generated by $$(\dots, a_{i-1}, e_i, a_i, \dots)\sim (\dots, a_{i-1}c, e_i, c^{-1}a_i, \dots)$$ for $c\in A_{e_i}$, $i=1,\dots, n$ and $$(\dots, e_{i-1}, 1, e_i,\dots)\sim (\dots, e_{i-2}, a_{i-2}a_i, e_{i+1}, \dots)$$ when $e_{i-1}= e_i\ii$. For an $\bbA$-path $t$, let $[t]$ denote its $\sim$-equivalence class. The \emph{fundamental group} $A=\pi_1(\bbA,v_0)$ of $\bbA$ is the group whose elements are $[t]$ for $\bbA$-paths $t$ whose underlying paths begin and end on $v_0$. Multiplication is given by concatenation.
\end{definition}

\begin{definition}[Reduced $\bbA$-path]
    An $\bbA$-path $(a_0, e_1, a_1, \dots, e_n, a_n)$ is \emph{reduced} if $a_{i}\notin (A_{e_{i}})_+$ whenever $e_{i+1}=e_i \ii$.
\end{definition}

\begin{remark}[Bass-Serre tree]
The \emph{Bass-Serre tree} $T_{\bbA}$ is a tree with an action $A\actson T_{\bbA}$ without edge-inversions such that $A\backslash T_\bbA \cong \bfA$ and the vertex (resp. edge) stabilizers in $A$ are isomorphic to the corresponding vertex (resp. edge) groups in $\bbA$. 


    Conversely, let $A$ be an arbitrary group acting on a tree $T$ without edge-inversions. Then the quotient $T/A$ can be given a graph of groups structure so that $T$ is the associated Bass-Serre tree.  
\end{remark}

\begin{convention} Throughout the paper all group actions on trees will be assumed to be without edge-inversions.
\end{convention}

\begin{definition}[Acylindricity]
    The graph of groups $\bbA$ is \emph{$k$-acylindrical} if the pointwise stabilizer of any path of length $k+1$ in $T_{\bbA}$ is trivial.
\end{definition}

\begin{definition}[Minimal]
    The action $A\actson T_\bbA$ is \emph{minimal} if there exists no $A$-invariant proper subtree. 
    Similarly, $A\actson T_\bbA$ is \emph{minimal relative to the vertices $\tild v_1,\dots ,\tild v_n$} if there exists no $A$-invariant proper subtree containing $\tild v_1,\dots,\tild v_n$.
    
    The graph of groups $\bbA$ is \emph{minimal} if the associated action $A\actson T_\bbA$ is minimal.
    Similarly $\bbA$ is \emph{minimal relative to the vertices $v_1,\dots,v_n$} if it is minimal relative to lifts $\tild v_1,\dots,\tild v_n$ of $v_1,\dots,v_n$ to $T_\bbA$.
\end{definition}

\begin{remark}\label{minimality criterion}
    Suppose $\bbA$ is a finite graph of groups. The action $A\actson T_\bbA$ is minimal if and only if  $(A_e)_+$ is a proper subgroup of $A_{e_+}$ whenever $e_+$ is a degree one vertex.
    Hence, by removing a sequence of degree one vertices $e_+$ with $A_{e_+}=(A_e)_+$ one obtains the graph of groups associated to the quotient of the minimal $A$-invariant subtree of the action $A\actson T$.
\end{remark}

\begin{definition}[Subgroup graph of groups]
    Given a subgroup $A'\le A$, we let $T_{A'}$ denote a minimal subtree of $A'\actson T_{\bbA}$ when it exists, and we let $\bbA'$ be the graph of groups structure on $T_{A'}/A'$ 
\end{definition}

\subsection{Relatively hyperbolic groups}

We briefly recall the definition of relatively hyperbolic groups relevant to us.
For a more thorough treatment see~\cite{Osin}.

Let $G$ be a group generated by a finite symmetric set $X$, and let $\bbP=\{P_1,\dots, P_n\}$ be a finite set of subgroups of $G$. 
Let $\calP$ be the disjoint union of copies of the $P_i\in \bbP$. 
We will consider the \emph{Cayley graph} $\Gamma = \Cay(G,X)$ of $G$ with respect to the generating set $X$, and the \emph{relative Cayley graph} $\widehat\Gamma=\Cay(G,X\cup \calP)$ of $G$ with respect to $X \cup \calP$.
\begin{definition}[$X$ and $X\cup\calP$-lengths]
    We denote by $|\cdot|_X$, $d_X$ and $|\cdot|_{X\cup \calP}$, $d_{X\cup\calP}$ the corresponding norms and distance functions.
For a path $s=(v_0,e_1,\dots,e_n,v_n)$ in $\widehat\Gamma$ define its \emph{$X$-length} by $\len_X(s) := \sum_{i=1}^n d_X(v_{i-1},v_{i})$.
\end{definition}

\begin{definition}[Edge labels]
    For an edge $e$ of $\widehat\Gamma$, we denote its \emph{label} by $\lab(e) \in X\cup \calP$. Similarly the label $\lab(p)$ of a path $p$ in $\widehat \Gamma$ is the product of labels of its edges. 
\end{definition}
\begin{definition}[Peripheral components]
    A \emph{$P_i$-path} is a path whose edge labels are in $P_i$.
    A \emph{$P_i$-component} of a path $p$ is a non-trivial maximal $P_i$-subpath of $p$.
    A vertex in the interior of a component is \emph{inner}, all other vertices are \emph{phase} vertices.
    Two $P_i$-paths are \emph{connected} if their endpoints are connected by an edge in $\widehat\Gamma$ with a label in $P_i$. 
    In other words, they are connected if they belong to the same coset of $P_i$. 
    A $P_i$-component of a path $p$ is \emph{isolated} if it is not connected to any other $P_i$-component. 
\end{definition}
\begin{definition}[Non-backtracking]
    The path $p$ is \emph{non-backtracking} if every $P_i$-component is isolated for each $i$.
\end{definition}

\begin{remark}\label{rem: removing backtracks}
    Every path $p$ can be replaced by a non-backtracking path $p'$ with the same endpoints: If $p=b_1qb_2$ where $q$ is a proper subpath of length at least two with endpoints on the same peripheral coset, then replacing $q$ with a single edge decreases the length of $p$. 
    This process stops at $p'$. 
    Additionally:
    \begin{itemize}
    \item We can write  $p=a_1p_1a_2\dots p_{n-1}a_n$ and $p'=a_1e_1a_2\dots e_{n-1}a_n$ where the $e_i$ are the peripheral edges of $p'$.
    \item If $p$ has no $P$-component with endpoints in $gP$, then neither does $p'$.
    \item If $p$ is $\lambda$-quasi-geodesic then so is $p'$.
    \item The vertex set of $p'$ is a subset of the vertex set of $p$.
    \end{itemize}
\end{remark}

\begin{definition}[Bounded Coset Penetration (BCP)]
\label{def:bounded coset penetration}
    The pair $(G,\bbP)$ satisfies the \emph{bounded coset penetration property} (BCP property) if for each $\lambda\geq 1$ there exists $\varepsilon =\varepsilon_{BCP}(\lambda, G, \bbP, X)$ such that if $p$, $q$ are non-backtracking $\lambda$-quasi-geodesics in $\widehat\Gamma$ with $p_-=q_-$ and $p_+=q_+$, then for any $P_i$-component $s$ of $p$, we have
    \begin{enumerate}
        \item $d_X(s_-,s_+)<\varepsilon$ if there exists no $P_i$-component of $q$ connected to $s$, or
        \item there exists a unique $P_i$-component $t$ of $q$ connected to $s$, and $d_X(s_-,t_-),d_X(s_+,t_+)<\varepsilon$. 
    \end{enumerate}
\end{definition}


\begin{definition}
\label{def:relatively hyperbolic}
    Let $X$ be a finite, symmetric generating set of $G$, and let $\bbP$ be a collection of subgroups of $G$. We say $G$ is \emph{hyperbolic relative to $\bbP$} if $(G,\bbP)$ satisfies the BCP property and $\widehat \Gamma$ is hyperbolic. 
    A \emph{parabolic subgroup} is a subgroup conjugate to an element $P\in \bbP$, and a \emph{peripheral coset} is a coset of the form $gP$ for some $g\in G, P\in \bbP$.
    The group $G$ is \emph{toral relatively hyperbolic} if it is torsion-free and hyperbolic relative to a set of finitely generated free abelian subgroups.
\end{definition}


We now introduce taut quasi-geodesics. Taut quasi-geodesics can backtrack yet still satisfy a variant of the BCP property (\Cref{upgraded BCP}). 

\begin{definition}[Taut quasi-geodesic]
    Let $\widehat \Gamma= \Cay(G, X\cup \calP)$ be the relative Cayley graph of a relatively hyperbolic group. A quasi-geodesic $s$ in $\widehat \Gamma$ is \emph{$\nu$-taut} if $\min\{|e|_X, |e'|_X\}<\nu$ for any two connected $P$-edges $e$, $e'$ of $s$. 
\end{definition}

It may be instructive to note the relationship between the definition of taut quasi-geodesic and Lemma 1.18(3) in Weidman-Weller~\cite{WeidWell}, but this is not essential to our presentation.

\begin{lemma}
\label{peripheral endpoint bounds}
    For every $\nu, \ell$ there exists $\alpha=\alpha(\nu ,\ell,G,\bbP,X)$ such that if $s$ is a $\nu$-taut path of length $\len(s)=\ell$ in $\widehat \Gamma$ and both endpoints of $s$ are in $gP$, for some $g\in G$ and $P\in \bbP$, then $$M(s)-\alpha \le |s|_X\le \len_X(s)\le M(s) + \alpha$$ where $M(s)$ is the maximum $X$-length of an edge in $s$ with both endpoints in $gP$ if such an edge exists, or $M(s)=0$ otherwise.
\end{lemma}

\begin{proof} 
    We prove the inequalities by induction on $\ell$. 
    Let $\ell >0$. Let $s$ be a path as in the lemma. We divide into two cases:

    \textbf{Case 1. The path $s$ has an edge with both endpoints in $gP$.} Let $e$ be such an edge of maximal $X$-length, i.e. $|e|_X=M(s)$. 
    Write $s=s_1 e s_2$. By the induction hypothesis there exists $\alpha'$ such that $\len_X(s_i)\le M(s_i)+\alpha'$.
    Every $P$-edge of $s$ with both endpoints in $gP$ is $P$-connected to $e$; since $s$ is $\nu$-taut, $M(s_i) \le \nu$. 
    Thus
    \begin{align*}
        |s|_X &\ge |e|_X - |s_1|_X - |s_2|_X\\
        &\ge |e|_X - \len_X(s_1)-\len_X(s_2)\\
        &\ge M(s) -2(\nu+\alpha'),
    \end{align*}
    and also
    \begin{align*}
        \len_X(s)&\le \len_X(s_1)+|e|_X+\len_X(s_2)\\
        &\le |e|_X+2(\nu+\alpha') \\
        &= M(s)+2(\nu+\alpha').
    \end{align*}

    \textbf{Case 2. The path $s$ does not have an edge with both endpoints in $gP$.} 
    In this case $M(s)=0$, so the lower bound $M(s)\le |s|_X$ holds trivially.
    For the upper bound, write $s = a_1s_1a_2\dots s_{n-1}a_n$ and $s'=a_1e_1a_2\dots e_{n-1}a_n$ as in \Cref{rem: removing backtracks}, 
    where the $e_i$ are the peripheral edges of the non-backtracking path $s'$.
    The path $s'$ is non-backtracking of length at most $\ell$, hence trivially an $\ell$-quasigeodesic.
    By the BCP property (\Cref{def:bounded coset penetration}) applied to $s'$ and the $P$-edge joining its endpoints, there exists $\varepsilon=\varepsilon_{BCP}(\ell,G,\bbP,X)$, such that $|e_i|_X\le \varepsilon$.
    By the induction hypothesis there exists $\alpha'$ such that $$\len_X(s_i)\le M(s_i)+\alpha'\le |s_i|_X+2\alpha' = |e_i|_X+2\alpha' \le \varepsilon +2\alpha'.$$
    Additionally, since $a_i$ contains no peripheral edges we have $\len_X(a_i) = \len(a_i).$
    By the above,  
    \begin{align*}
        \len_X(s) &\le \sum _{i=1}^n\len_X (a_i) + \sum_{i=1}^{n-1}\len_X(s_i)\\
        &\le \len(s) + n (\varepsilon+2\alpha') \\
        &\le \ell(1+\varepsilon +2\alpha')\\
        &=M(s)+\ell(1+\varepsilon +2\alpha')
    \end{align*} 

    In both cases we have the desired inequalities $$M(s)-\alpha\le |s|_X\le \len_X(s)\le M(s)+\alpha$$ for $\alpha  = \max\{2(\nu+\alpha'),\ell(1+\varepsilon +2\alpha')\}$.
\end{proof}

\begin{lemma}[Taut BCP property]\label{upgraded BCP}
    Let $G$ be hyperbolic relative to $\bbP$ with finite, symmetric generating set $X$. There exists $\varepsilon=\varepsilon(\lambda, \nu, G, \bbP, X)$ such that the following holds: 
    
     Let $p$, $q$ be $\nu$-taut $\lambda$-quasi-geodesics in $\widehat \Gamma$ with $p_-=q_-$ and $p_+=q_+$. For $g\in G$, $P\in \bbP$, let $p'$ be the maximal subpath of $p$ with $p'_-, p'_+\in gP$. Either $\len_X(p')\le 2\varepsilon$, or $p'$ contains an edge $e$ with $e_-, e_+\in gP$ and $|e|_X\geq \varepsilon$. In the latter case, $q$ contains an edge $e'$ with $e'_-, e'_+\in gP$ and $d_X(e'_-, e_-), d_X(e'_+, e_+)<\varepsilon$.
\end{lemma}

\begin{proof}
    Set $\varepsilon=3(\nu+\alpha+\varepsilon_{BCP})$ with $\alpha=\alpha(\nu,\lambda^2+\lambda,G,\bbP,X)$ from \Cref{peripheral endpoint bounds} and $\varepsilon_{BCP}=\varepsilon_{BCP}(\lambda, G, \bbP, X)$ from \Cref{def:bounded coset penetration}.
    Let $p$ and $p'$ be as in the statement of the lemma.
    Since $p'$ is $\lambda$-quasigeodesic and $d(p'_-, p'_+)=1$, we have $\len(p')\le \lambda^2+\lambda$. 
    Thus $\len_X(p')\leq M(p')+\alpha$ by \Cref{peripheral endpoint bounds}. 

    If $|e|_X\le\varepsilon$ for each $P$-edge of $p'$ in $gP$, then this bound becomes $\len_X(p')\le\varepsilon+\alpha\le 2\varepsilon$. 

    Otherwise, by $\nu$-tautness, there exists a unique $P$-edge $e$ of $p'$ with $|e|_X>\varepsilon$ since $\varepsilon>\nu$. Write $p'=p_1ep_2$. 
    As in the previous case, $\len_X(p_1), \len_X(p_2) < \nu+\alpha$. Since $|e|_X>\varepsilon$, then the $P$-edge $\bar e$ joining the endpoints of $p'$ satisfies $$|\bar e|_X\ge |e|_X-\len_X(p_1)-\len_X(p_2)> 3\varepsilon_{BCP}+\alpha+\nu.$$
    
    Let $\hat p$ be the path obtained from $p$ by replacing $p'$ with $\bar e$. 
    Note $\hat p$ is still $\nu$-taut. 
    We claim no subpath $p''$ of $\hat p$ containing $\bar e$ has both endpoints on the same peripheral coset: 
    If there exists $p''\supset \bar e$ with its endpoints on some $g'P'$, then we can choose such $p''$ so that only its endpoints meet $g'P'$.  
    Then by \Cref{peripheral endpoint bounds} applied to $p''$ and $g'P'$, we obtain $|\bar e|_X\leq \len_X(p'')\leq \alpha$, a contradiction. 
    Thus the path $\bar p$ obtained from $p$ as in \Cref{rem: removing backtracks} can be taken to contain $\bar e$. 

    Let $\bar q$ be obtained from $q$ as in \Cref{rem: removing backtracks}. 
    By the BCP property, $\bar q$ contains an edge $\bar e'$ in $gP$ with $d_X(\bar e'_-, \bar e_-), d_X(\bar e'_+, \bar e_+)<\varepsilon_{BCP}$ and thus $|\bar e'|_X>\varepsilon_{BCP}+\alpha+\nu$. 
    Let $q'$ be the subpath of $q$ replaced by $\bar e'$. 
    Since $q'$ is $\lambda$-quasigeodesic and $d(q'_{-},q'_+) = 1$, we have $\len(q')\le \lambda^2+\lambda$.
    By \Cref{peripheral endpoint bounds} applied to $q'$ and $gP$, $$\varepsilon_{BCP} +\nu+\alpha< |\bar e'|_X=|q'|_X\le M(q')+\alpha $$
    Therefore, $M(q')> \varepsilon_{BCP}+\nu$.
    By $\nu$-tautness, there exists a unique $P$-edge $e'$ of $q'$ with endpoints in $gP$ such that $|e'|_X=M(q')$. Write $q'=q_1e'q_2$.
    By \Cref{peripheral endpoint bounds} applied to each $q_i$ and $gP$, $|q_i|_X\le \nu+\alpha$.
    We have
    $$d_X(e_-, e'_-)\le |p_1|_X + \varepsilon_{BCP} + |q_1|_X  = \varepsilon_{BCP}+2(\nu + \alpha) <\varepsilon, $$
    and similarly $d_X(e_+, e'_+)<\varepsilon$.
\end{proof}

\subsection{Relatively quasiconvex subgroups}
\begin{definition}[Relative quasiconvexity]
Let $(G,\bbP)$ be relatively hyperbolic. A subgroup $H\le G$ is $\nu$-\emph{relatively quasiconvex} if the phase vertices of any geodesic path in $\widehat\Gamma$ between vertices in $H$ are in the $\nu$-neighborhood of $H$ with respect to the metric $d_X$.
\end{definition}

\begin{theorem}
\label{thm:induced structure}
    Let $H$ be a relatively quasiconvex subgroup of a relatively hyperbolic group $(G,\bbP)$. Then there exists a finite, symmetric generating set $Y$ of $H$, and the set
    $$\{H\cap P_i^g \;|\; g\in G, P_i\in \bbP, |H\cap P^g|=\infty\}$$
    consists of finitely many $H$-conjugacy classes of subgroups of $H$. For any set of representatives $\bbO$, the pair $(H,\bbO)$ is relatively hyperbolic, and the inclusion $(H,d_{Y\cup \calO})\to (G,d_{X\cup \calP})$ is a quasi-isometric embedding. 
\end{theorem}

\begin{definition}[Induced peripheral structure]
    We refer to a pair $(H,\bbO)$ as in \Cref{thm:induced structure} as a \emph{$(G,\bbP)$ induced peripheral structure} on $H$. 
\end{definition}

\begin{lemma}\label{peripheral coset distance in quasiconvex}
    Let $H$ be a $\nu$-relatively quasiconvex subgroup of $(G,\bbP)$. 
    There exists $\rho = \rho(H)$ such that the following holds:
    Let $\gamma$ be a geodesic in $\hat \Gamma$ representing an element in $H$.
    Assume $\gamma$ is labeled by $\tau p \tau'$ where $p\in P\in \bbP$.  
    Then, $d_X(Qp,Q)<\rho$ where $Q = P  \cap \tau\ii H\tau$.
\end{lemma}

\begin{proof}
    Consider a sequence of geodesics $\gamma_n$ representing elements $h_n \in H$,  labeled by the concatenation $\tau_n p_n \tau'_n$ for $p_n\in P_n\in \bbP$. Assume for the sake of contradiction that $d_X(Q_np_n,Q_n) \to \infty$ where $Q_n = P_n\cap \tau_n\ii H \tau_n$. 
    
    Since there are only finitely many peripheral subgroups in $\bbP$, by passing to a subsequence, we may assume that there is $P\in \bbP$ such that  such that $P_n = P$ for all $n$.

    
    By definition of $\nu$-relative quasiconvexity, the element $\tau_n$ is $\nu$-close in $X$-distance to some $x_n\in H$. Similarly, the element $\tau_n p_n$ is $\nu$-close to some $x'_n\in H$.
    Consider the geodesic $\gamma'_n$ connecting $x_n$ to $x_n'$.
    Its label represents the element $x_n\ii x_n'\in H$.
    Since $|p_n|_X\to \infty$, by BCP, the geodesic $\gamma'_n$ has an edge $e'_n$ connected to the edge $e_n$ of $\gamma_n$ labeled by $p_n$, and $d_X(e_{n-},e'_{n-}),d_X(e_{n+},e'_{n+})<\varepsilon$. 
    Therefore, the label $p'_n$ of $e'_n$ satisfies 
    $$d_X(Q_np_n,Q_n) \le d_X(Q_np'_n,Q_n)+2\varepsilon.$$
    Since $x_n\in H$ we have  $Q_n = P  \cap \tau_n\ii H\tau_n = P\cap (x_n\ii \tau_n)\ii  H (x_n\ii\tau_n)$. 
    Moreover, using $\tau_n=e_{n-}$ we have 
    $$d_X(x_n,e'_{n-})\le d_X(x_n,e_{n-})+d_X(e_{n-},e'_{n-})\le \nu + \varepsilon,$$
    and similarly,
    $$d_X(x_n',e'_{n+})\le \nu + \varepsilon.$$

    Thus up to replacing $\gamma_n$ with $\gamma'_n$, we may assume that $|\tau_n|_X,|\tau'_n|_X\le \nu+\varepsilon$.
    
    By passing to a subsequence, we may assume that there are elements $\tau,\tau'$ such that $\tau = \tau_n,\tau'=\tau_n'$ for all $n$.
    Note that $Q_n=Q := P\cap \tau\ii H \tau$ for all $n$. 
    For all $m,n$, we have $Qp_n = Qp_m$ since $$P\ni p_n p_m\ii = \tau\ii h_n h_m\ii \tau \in \tau\ii H \tau.$$ 
    However, this contradicts the assumption $d_X(Qp_n,Q)\to \infty$.
\end{proof}

\section{A User's Guide to Weidmann-Weller Carrier Graphs}
\label{sec: users guide to weidmann weller}


Carrier graphs, as defined by Weidmann-Weller \cite{WeidWell}, are a way of recording the peripheral structure of subgroups of relatively hyperbolic groups. 
Informally, they are graphs of groups $\bbA$ together with distinguished subgraphs, called \emph{peripheral stars}, and a map $\nu_\calA: \pi_1(\bbA,v_0)\to G$ 
which maps the peripheral stars into the peripheral cosets of $G$.

\begin{definition}[Carrier graph~\cite{WeidWell}*{Definition 2.6}, $P_i$-carrier star~\cite{WeidWell}*{Definition 2.5}]
    A $(G,\bbP)$-carrier graph of groups is a tuple \[\calA=(\bbA, (g_e)_{e\in E\bfA}, (\bbC_j,c_j)_{1\leq j\leq k}, v_0)\] where $\bbA$ is an oriented graph of groups, $g_e\in G$, each $\bbC_j$ is the full subgraph of groups whose underlying graph $\bfC_j$ is a star consisting of all edges incident to $c_j\in V\bfA$, and $v_0 \in V\bfA$ is the basepoint such that:
    \begin{enumerate}[label = (A\arabic*)]
        \item $A_v, A_e\le G$ for each $v\in V\bfA$ and $e\in E\bfA$. 
        \item $g_{e^{-1}}=(g_e)^{-1}$ for each $e\in E\bfA$. 
        \item For each oriented edge $e$, $A_e\to A_{e_-}$ is an inclusion of subgroups of $G$, and $A_e\to A_{e^+}$ is given by $a\mapsto g_e^{-1}ag_e$.
        \item For each $j=1,\dots, k$, there exists $P\in \bbP$ so that $C_v, C_e\le P$ for each $v\in V\bfC_j$ and $e\in E\bfC_j$, and $g_e\in P$ for each $e\in E\bfC_j$.  
        \item For each $j=1,\dots, k$ and oriented edge $e\in E\bfC_j$, we have that $e_-=c_j$, and the morphisms $A_e\to A_{e_-}$ and $A_e\to A_{e_+}$ are isomorphisms.
        \item\label{carrier: non-trivial turns} $g_eg_{e'}^{-1}\not\in A_{c_j}$ for each distinct pair of oriented edges $e,e'\in \bfC_j$. 
        \item $\bfC_i\cap \bfC_j=\emptyset$ for $i\neq j$. 
        \item \label{carrier: non free edges} For each oriented edge $e\in E\bfA$ with $A_e\neq 1$, we have $e_-\in V\bfC_j$ for some $j=1,\dots, k$ and $e_+\not\in V\bfC_i$ for $i\neq j$.
        \item 
        \label{WW: carrier is minimal}
        \label{carrier: minimality} 
        $\bbA$ is minimal relative to the base point and the non-trivial central peripheral vertices, i.e. for every $v \in V\bfA$ with $\text{val}(v) = 1$ and surjective boundary morphism $A_e \rightarrow A_v$, we have $v = v_0$ or $v = c_j$ for $1 \leq j \leq k$ with $C_{c_j} \neq 1$.
    \end{enumerate}
    
    For any $\mathbb{A}$-path $t = (a_0, e_1, a_1, \dots, e_k, a_k)$, define:
    \[\nu_{\mathcal{A}}(t) := a_0 g_{e_1} a_1 \cdot \ldots \cdot g_{e_k} a_k.\]
    This induces a well-defined group homomorphism:
    \[\nu_{\mathcal{A}}: \pi_1(\mathbb{A}, v_0) \to G, \quad [t] \mapsto \nu_{\mathcal{A}}([t]) := \nu_{\mathcal{A}}(t).\]
    Each $(\bbC_j,c_j)$ is a \emph{peripheral star}, and vertices or edges contained in some $V\bfC_j$ or $E\bfC_j$ are \emph{peripheral}. A non-peripheral edge $e\in E\bfA$ (respectively, vertex $v\in V\bfA$) is \emph{essential} if $A_e\neq 1$ (respectively, $A_v\neq 1$). Otherwise, a non-peripheral edge (respectively, vertex) is \emph{free}.
\end{definition}

\begin{example}
     The group $G=\gen{a,b,c,d|c=[a,b],[c,d]=1}$ which splits as $F_2*_\bbZ \bbZ^2=\gen{a,b}*_{[a,b]=c}\gen{c,d}$ is hyperbolic relative to $P=\gen{c,d}$.
     The subgroup $$H=\gen{a^2,b^2,ab, \quad a^{d^5}, b^{d^5},\quad (a^2b^3)^{d^{100}}}$$ is the image of $\nu_\calA$ for the carrier graph depicted in \Cref{fig:carrier_graph}.
     \begin{figure}[h]
         \centering
         \includegraphics{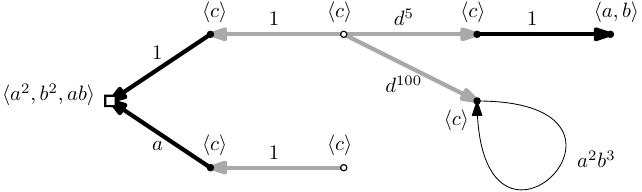}
         \caption{Peripheral stars have central vertices marked by $\circ$ and gray edges.
         The base vertex $v_0$ is marked by $\square$.
         The edge labels and vertex groups are marked on the graph.
         Edge groups of thick edges are all $\gen{c}$, and the thin edge is free. }
         \label{fig:carrier_graph}
     \end{figure}
\end{example}

\begin{remark}\label{rem: peripheral turns are realizable}
    As a consequence of \ref{carrier: minimality}, for any two distinct peripheral edges $e$, $e'$ with $e_-=e'_-$, there exists a reduced $\bbA$-path $t\in [t]\in \pi_1(\bbA, v_0)$ whose underlying path has the form $(\dots, e \ii, e', \dots)$. Moreover, if $v_0 = e_+$ then the path can be taken to be of the form $(e\ii,e',\dots)$.
\end{remark}

\begin{definition}[Peripheral structure]
Let $\calA$ be a $(G,\bbP)$-carrier graph with basepoint $v_0\in V\bfA$. For each peripheral star $(\bbC_j,c_j)$ with $A_{c_j}\neq 1$, choose an $\bbA$-path $g_j$ whose underlying path starts at $v_0$ and ends at $c_j$. A \emph{peripheral structure} $\bbO_\calA$ is the collection of conjugates $\nu_\calA(g_j) A_{c_j}\nu_\calA(g_j)\ii$ for some such choice of $\bbA$-paths.
\end{definition}

It is shown in \cite{WeidWell} (see \Cref{weidmann weller rewrite} below) that a locally quasiconvex subgroup $H$ of $G$ has a particularly nice carrier graph $\calA$ so that 
\begin{itemize}
    \item the induced peripheral structure of $H$ is encoded by the peripheral structure of $\calA$,
    \item there exists a metric on $\pi(\bbA,v_0)$ (see \Cref{definition:metric on carriers}) so that $\nu_\calA:\pi_1(\bbA,v_0)\to G$ is a $C$-quasi-isometric embedding, and
    \item all essential vertex groups are $\nu$-relatively quasiconvex
\end{itemize}
where both constants $C,\nu$ depend only on $G$ and $\rank(H)$.
Moreover, for locally quasiconvex subgroups of a bounded rank, there are only finitely many possibilities for the essential portions of their carrier graphs.
Consequently, for the right notion of equivalence (see \Cref{definition:equivalence of carrier graphs}), the carrier graphs of locally quasiconvex subgroups of a bounded rank fall into finitely many equivalence classes.

\begin{definition}[Realizations]\label{realizations}
    Let $\calA$ be a $(G,\bbP)$-carrier graph. A \emph{realization} $\bar s$ of the $\bbA$-path $t=(a_0,e_1,a_1,\dots,e_k,a_k)$ is any path constructed as follows: First, consider a path $s= \alpha_0 \epsilon_1\alpha_1\dots\epsilon _k\alpha_k$ in $\widehat \Gamma$ such that for all $i$, $\alpha_i$ is a geodesic representing $a_i$, and $\epsilon_i$ is a geodesic representing $g_{e_i}$.
    Then form $\bar s$ by replacing each subpath of $s$ arising from a maximal $\bbC_i$-subpath of $t$ by a single edge with label in $\calP$ connecting its endpoints.
\end{definition}

\begin{definition}[$|t|_{X\cup \calP}^\bbA$ and $d_{X\cup\calP}^\bbA$]
    \label{definition:metric on carriers}
    For any $\bbA$-path $t$, let $|t|_{X\cup \calP}^\bbA$ be the minimal length of a realization of some $t'\in [t]$. This defines a metric on $\pi_1(\bbA, v_0)$ by $d_{X\cup \calP}^\bbA([g],[h])=|gh^{-1}|_{X\cup\calP}^\bbA$.
\end{definition}

\begin{definition}[Equivalence of carrier graphs]
    \label{definition:equivalence of carrier graphs}
    Let $\calA,\calB$ be $(G,\bbP)$-carrier graphs. Let $\bfA',\bfB'$ be the (possibly disconnected) subgraphs of $\bfA,\bfB$ consisting of all non-trivial vertices and edges in $\calA,\calB$ respectively.
    The carrier graphs $\calA,\calB$ are \emph{equivalent} if $b_1(\bfA) = b_1(\bfB)$, and there exists a graph isomorphism $f:\bfA'\to \bfB'$ such that:
    \begin{enumerate}
            \item The map $f$ sends the (non-trivial) peripheral stars of $\bfA'$ isomorphically to the peripheral stars of $\bfB'$.
            \item For every essential vertex $v\in V\bfA$, $A_v = B_{f(v)}$.
            \item For every essential edge $e\in E\bfA'$, $A_e = B_{f(v)}$ and $g_e = g_{f(e)}$.
    \end{enumerate}
\end{definition}

\begin{theorem}[Weidmann-Weller \cite{WeidWell}]
\label{weidmann weller rewrite}
Let $G$ be torsion-free and hyperbolic relative to a collection of finitely generated abelian subgroups $\mathbb{P} = \{P_1,\dots,P_n\}$. Let $X$ be a finite, symmetric generating set of $G$. 

For $r\in\bbN$ there exist constants $C$, $\nu$ and equivalence classes $\frC_1,\dots,\frC_N$ of $(G,\bbP)$-carrier graphs such that the following holds: 

If $H\le G$ is a locally relative quasi-convex subgroup of $\rank(H)\le r$ then there exists a $(G,\bbP)$-carrier graph $\calA = (\bbA,(g_e)_{e\in E\bfA},(\bbC_i,c_i)_{1\leq i\leq k}, v_0)$ such that:
\begin{enumerate}[label = (C\arabic*)]
    \item \label{WW: image is H} $\im(\nu_\calA)=H$.
    \item \label{WW: finitely many equivalences} The carrier graph $\calA$ belongs to one of the equivalence classes $\frC_1,\dots,\frC_N$. 
    
    \item \label{WW: paths and qi} For any $\bbA$-path $t'$, there exists a realization $\bar s$ of some reduced $t\in [t']$, 
    such that $\bar s$ is a $\nu$-taut $C$-quasi-geodesic.
    In particular, the map $\nu_\mathcal{A} :(\pi_1(\mathbb{A},v_0), d^\mathbb{A}_{X \cup \mathcal{P}}) \to (G, d_{X \cup \mathbb{P}})$ is an injective $C$-quasi-isometric embedding.
    \item \label{WW: peripheral} For any peripheral structure $\mathbb{O}_\mathcal{A}$ of $\mathcal{A}$, the image $\nu_\mathcal{A}(\mathbb{O}_\mathcal{A})$ is a $(G,\bbP)$ induced peripheral structure on $H$. 
    \item \label{WW: prenormal} For every essential vertex $v$ the subgroup $A_v$ is:
    \begin{enumerate}[label=(\roman*)]
        \item $\nu$-relatively quasi-convex
        \item The collection $\bbO_v = \{(A_e)_- \;|\;e\in E\bfA, e_-=v,A_e \ne 1  \}$ is a $(G,\bbP)$ induced peripheral structure on $A_v$.
    \end{enumerate}
    \item \label{WW: prenormal edge length} $|g_e|_X\leq \nu$ for each essential $e\in E\bfA$.
    \item \label{WW: no P component in free edges} For every free edge $e\in E\bfA$ \begin{enumerate}[label = (\roman*)]
    \item \label{WW: reduced free edges} If an endpoint of $e$ is  free, then $g_e\ne 1$.
    \item \label{WW: converse of Condition (A8)} The geodesic paths representing $g_e$ have no $P_i$-components. 
    \end{enumerate}
    
\end{enumerate}
\end{theorem}

\begin{proof}
    This theorem consists mainly of Theorem 4.10 in \cite{WeidWell} together with some notions defined explicitly there. 
    Theorem 4.10 of \cite{WeidWell} shows that there exists a $M$-(pre)normal $(G,\bbP)$-carrier graph that fulfills \ref{WW: image is H}, \ref{WW: finitely many equivalences} and \ref{WW: peripheral}.  
    Items \ref{WW: prenormal} and \ref{WW: prenormal edge length} are the notion of $M$-prenormal $(G,\bbP)$-carrier graphs used in \cite{WeidWell}: the proof of Theorem 4.10 relies on Proposition 4.9, which produces an $M$-(pre)normal $(G,\bbP)$-carrier graph in an equivalence classes of $M$-(pre)normal $(G,\bbP)$-carrier graphs. The result of this application then (possibly) has a fold applied to it, Lemma 4.7 of \cite{WeidWell} guarantees the folded graph is again (pre)normal. 
    Except for \ref{WW: no P component in free edges}\ref{WW: reduced free edges}, the remaining properties are consequences of Proposition 3.1 and Corollary 3.6.
    Proposition 3.1 of \cite{WeidWell} implies \ref{WW: paths and qi}; indeed, in the course of the proof of Proposition 3.1, it is shown every subpath satisfies Lemma 1.18(3), the converse of Lemma 1.18(3) is the definition of $\nu$-taut.
    Item \ref{WW: no P component in free edges}\ref{WW: converse of Condition (A8)} is the converse of the condition \ref{carrier: non free edges} from Corollary 3.6 in \cite{WeidWell}.
    As Proposition 3.1 and Corollary 3.6 are used in the proof \ref{WW: paths and qi} and \ref{WW: no P component in free edges}\ref{WW: converse of Condition (A8)} also hold.

    For \ref{WW: no P component in free edges}\ref{WW: reduced free edges} assume that there exists a free edge $e$ with $g_e = 1$. If $e$ is a loop, then the map $\nu_\mathcal{A}$ is not injective, which contradicts \ref{WW: paths and qi}.

    Therefore, $e$ is an edge with endpoints $v,v'$ with $v \neq v'$. Now, we can delete $e$ and identify the vertices $v$ and $v'$. We get a new vertex $v''$ with the vertex group $A_{v''} = A_{v} \cup A_{v'}$, which is a group as $A_{v} = 1$ or $A_{v'} = 1$. This reduction step does not interfere with any of the other conditions. As there are only finitely many edges, after a finite number of reduction steps we get a $(G,\bbP)$-carrier graph fulfilling \ref{WW: no P component in free edges}\ref{WW: reduced free edges}.
\end{proof}


\section{Ascending chains in toral relatively hyperbolic groups}
\label{sec: omegaACC for toral rel hyp}
In this section we prove \Cref{main theorem rel Hyperbolic} and deduce $\omega$ACC for complete finite-volume hyperbolic 3-manifold groups.

\begin{reptheoremA}{main theorem rel Hyperbolic}
    In a toral relatively hyperbolic group, any ascending chain of bounded rank, locally relatively quasiconvex subgroups stabilizes. 
\end{reptheoremA}

The following allows us to reduce to ascending chains such that no subgroup is a proper free factor of a subgroup later in the sequence.

\begin{lemma}\label{free factor reduction}
Let $H_1\le H_2\le\dots$ be an ascending chain of subgroups of $G$ with uniformly bounded rank.
There exists an ascending chain $K_1\le K_2\le \dots$ of subgroups of $G$ with uniformly bounded rank such that $K_{i}$ is not contained in a proper free factor of $K_{j}$ for any $i\leq j$, and $K_1\le K_2\le \dots$ stabilizes if and only if $H_1\le H_2\le \dots$ stabilizes. Moreover, each $K_i$ is a subgroup of some $H_j$.
\end{lemma}

\begin{proof}
By induction on $n=\sup_i\rank H_i$. If $n=1$, then the $H_i$ are all cyclic and contain no proper free factors, so we suppose $n>1$. 
If for all sufficiently large integers $i<j$, $H_i$ is not contained in a proper free factor of $H_j$, then we can take $K_1\le K_2\le\dots$ to be a tail of $H_1\le H_2\le \dots$.

Otherwise, there exists an ascending sequence $i_1<j_1<i_2<j_2<\dots$ such that $H_{i_k}$ is contained in a proper free factor $L_k$ of $H_{j_k}$ for all $k=1,2,\dots$. We have 
\[H_{i_{1}} \le L_1 \le H_{j_{1}} \le H_{i_2} \le L_2\le H_{j_{2}} \le \dots,\]
so $L_1\le L_2\le\dots$ stabilizes if and only if $H_1\le H_2\le \dots$ stabilizes. 
Each $L_k$ is a proper free factor of $H_{j_k}$, so $\rank L_k<\rank H_{j_k}\leq n$.
By the induction hypothesis applied to $L_1\le L_2\le \dots $, there exists a sequence $K_1\le K_2\le\dots $ such that $K_i$ is not contained in a free factor of $K_j$ for all $i<j$, and $K_1\le K_2\le \dots $ stabilizes if and only if $L_1\le L_2\le \dots $ stabilizes. Thus, $K_1\le K_2\le \dots$ stabilizes if and only if $H_1\le H_2\le \dots $ stabilizes.
Moreover, each $K_i$ is a subgroup of some $L_j$ which is itself a subgroup of some $H_k$.
\end{proof}

\begin{proof}[Proof of \Cref{main theorem rel Hyperbolic}]
Let $(G,\mathbb{P})$ be a toral relatively hyperbolic group, and let $H_1\le H_2\le\dots $ be an ascending chain of locally quasiconvex subgroups of uniformly bounded rank.

By \Cref{free factor reduction} we can assume that $H_{i}$ is not contained in a proper free factor of any of the $H_{j}$ for all $i<j$. 
By \Cref{weidmann weller rewrite}, there exist constants $C,\nu$ and finitely many equivalence classes such that for each subgroup $H_i$ there is $(G,\bbP)$-carrier graph $$\calA_i=(\bbA_i, (g_e)_{e\in E\bfA_i},(\bbC_{j_i},c_{j_i})_{1\leq j_i\leq k_i},v_0)$$ in one of those classes satisfying properties \ref{WW: image is H}-\ref{WW: no P component in free edges}. By passing to a subsequence, we can suppose the  $\calA_i$ all belong to a single equivalence class $\frC$.

The proof proceeds in steps. At each step, we prove there are finitely many options for various aspects of the carrier graphs' data. 
By passing to subsequences, we may assume that these data are fixed. 
Eventually, all the data is fixed, demonstrating that the sequence stabilizes.

\bigskip

Let $h_1,\dots,h_n$ be the generators of $H_1$. 
Let $\gamma_1,\dots,\gamma_n$ be geodesics in $\widehat \Gamma$ representing $h_1,\dots,h_n$, and let
$$M =  \max_i \{\len_X(\gamma_i)\}.$$
Note that $|h_i|_{X\cup \calP}\le M $ for $i=1,\dots,n$.

For $i\ge 1$, since $H_1\le H_i$, the elements $h_1,\dots,h_n$ can be represented by $\bbA_i$-paths $t_{i,1},\dots,t_{i,n}$ whose realizations $\bar s_{i,1},\dots,\bar s_{i,n}$ are $\nu$-taut $C$-quasigeodesics by \ref{WW: paths and qi}.
Denote $T_i = \{t_{i,1},\dots,t_{i,n}\}$ and $S_i = \{\bar s_{i,1},\dots,\bar s_{i,n}\}$.

\bigskip

\begin{step} 
The number of edges in $\bfA_i$ and the $X$-lengths of all free edges are uniformly bounded, i.e. $$\sup_i| E\bfA_i|<\infty \quad \text {and }\quad \sup\{|g_e|_X\;:\;i\geq 0,\;e\in E\bfA_i\text{ is free}\}<\infty.$$
\end{step}

\begin{claimproof}[Proof of Step 1]
Let $i\ge 1$. 
The carrier graph $\calA_i$ belongs to the equivalence class $\frC$ and satisfies \ref{WW: carrier is minimal}. 
Thus the underlying graph $\bfA_i$ belongs to one of finitely many homeomorphism classes. Moreover, the number of non-trivial vertices and edges in $\bfA_i$ is determined by $\frC$. 
Therefore, to bound the number of edges in $\bfA_i$ it remains to bound $\len(p)$ for embedded paths $p=(v_0,e_1,\dots, e_l, v_l)$ in $\bfA_i$ with $A_{v_j}$ trivial and $\deg(v_j)=2$ for $j=1,\dots, l-1$.

Since $H_1\le H_i$ is not contained in a proper free factor and the path $p$ gives rise to a non-trivial free splitting of $H_i$, there exists $h\in \{h_1,\dots,h_n\}$ such that any $\bbA_i$-path $t$ representing $h$ contains $p$ as a subpath.
Let $\bar s \in S_i$ be the $\nu$-taut $C$-quasigeodesic realization of the $\bbA_i$-path $t\in T_i$ representing $h$. The realization $\bar s$ is a $C$-quasigeodesic, so $\len(\bar s) \le CM+C^2$. The realization $\bar s$ contains the realization $\bar q$ of $p$ as a subpath. 
The subpath $\bar q$ is obtained in two steps: 
first, construct the path $q$ by concatenating geodesics representing each $g_{e_j}$; then, replace each $\bbC_i$ subpath in $q$ by a single edge to obtain $\bar q$. 
Each peripheral edge $e$ of $p$ belongs to a $\bbC_i$ subpath of length two, unless $e$ is the first or last edge of $p$. 
Additionally, by \ref{carrier: non-trivial turns} and 
\ref{WW: no P component in free edges}\ref{WW: reduced free edges} for any pair of consecutive edges $e$, $e'$ in $p$ with $g_e=g_{e'}=1$, exactly one of $e$, $e'$ is peripheral. 
Thus the piecewise $X$-length of any four consecutive edges of $p$ is at least one. 
Together, these observations imply $$\len(p) \le 4\len (\bar q)+3 \le 4\len(\bar s)+3\le 4(CM+C^2)+3.$$

By \ref{WW: no P component in free edges}, for every free edge $e$ in $p$, the geodesic representing $g_e$ has no $P_i$-component. So, by the same reasoning \[|g_e|_X =|g_e|_{X\cup \calP} \le \len(\bar q) \le CM+C^2.\qedhere \]
\end{claimproof}

By Step 1 and the definition of carrier equivalence, up
to passing to a subsequence, we may assume that the the underlying graph of $\calA_i$ stabilizes, and so does the non-peripheral structure of $\calA_i$. 
That is,
there exists a $(G,\bbP)$-carrier graph $(\bbA,(g_e)_{e\in E\bfA}, (\bbC_j,c_j)_{1\le j\le k},v_0)$ 
and graph isomorphisms $f_i:\bfA\to \bfA_i$ such that for each $i=1,2,\dots$ 
\begin{enumerate}[label = (P\arabic*)]
    \item \label{basepoint} $f_i$ carries the basepoint of $\calA$ to the basepoint of $\calA_i$.
    \item $f_i$ carries peripheral stars bijectively to peripheral stars.
    \item $A_v = A_{f_i(v)}$ for non-peripheral $v\in V\bfA$.
    \item $A_e = A_{f_i(e)}$ and $g_e = g_{f_i(e)}$ for each non-peripheral edge $e\in E\bfA$ and all $i=1,2,\dots$.
\end{enumerate}

\begin{step}
There is a uniform bound on $|g_e\ii g_{e'}|_X$ for any pair of peripheral edges $e$, $e'$ such that $e_-=e'_-$ is the center of a trivial peripheral star in some $\calA_i$.
\end{step} 

\begin{claimproof}[Proof of Step 2]
Let $i\ge 1$. Let $(\bbC, c)$ be a trivial peripheral star in $\calA_i$. Consider the graph $\bfL_i$ with vertices $\{ e\in E\bfC\;|\;e_-=c\}$, where $e$, $e'$ are joined by an edge if some $\bbA_i$-path $t\in T_i$ representing a generator $h\in \{h_1,\dots, h_n\}$ contains the subpath $(e^{-1},1, e')$ or the reverse subpath $(e'^{-1},1,e)$. 

The graph $\bfL_i$ is connected: Indeed, 
any non-trivial partition $V\bfL_i= A\sqcup B$, induces a free splitting of $H_i$. 
One of the free factors in this splitting consists of all closed $\bbA_i$-paths based at $v_0$ without a subpath $(e^{-1}, 1,e')$ where 
$e\in A$ and $e'\in B$, or $e\in B$ and $e' \in A$.
Since $H_1\le H_i$ is not contained in a proper free factor, some $\bbA_i$-path $t\in T_i$ is not in the free factor described above. 
In other words, there is an edge in $\bfL_i$ joining a vertex in $A$ to a vertex in $B$. 

By Step 1, the number of vertices in $\bfL_i$ is uniformly bounded. In particular, for any $e,e'\in V\bfL_i$, there exists a sequence $e = e_0,\dots,e_k=e'\in V\bfL_i$ of adjacent vertices in $\bfL_i$ where $k$ is uniformly bounded independent of $i$. We have $$|g_e\ii g_{e'}|_X \le |g_{e_0}\ii g_{e_1}|_X+\dots +|g_{e_{k-1}}\ii g_{e_k}|_X.$$
Thus, it suffices to bound $|g_e\ii g_{e'}|$ for edges $(e,e')\in E\bfL_i$. 

Let $t\in T_i$ be an element representing $h\in \{h_1,\dots,h_n\}$ such that $(e\ii,1,e')$ is a $\bbC$-subpath of $t$. 
The subpath $(e\ii,1,e')$ is replaced by a single edge $\bar e$ when forming $\bar s$, so $|g_e\ii g_{e'}|_X = d_X(\bar e_-,\bar e_+)$.
Since $\bar{s}$ is a $\nu$-taut $C$-quasigeodesic, we apply \Cref{upgraded BCP} and deduce: either $d_X(\bar e_-,\bar e_+)<2\varepsilon$, or the endpoints $\bar e_-,\bar e_+$ are $\varepsilon$-close in $X$-distance to points $x_-,x_+$ on the fixed geodesic $\gamma$ representing $h$. 
Thus,
\[ |g_e\ii g_{e'}|_X = d_X(\bar e_-,\bar e_+)\le d_X(\bar e_-,x_-)+d_X(x_-,x_+)+d_X(x_+,\bar e_+)\le 2\varepsilon +M. \qedhere\]
\end{claimproof}

A peripheral star $(\bbC, c)$ of $\calA$ is \emph{controlled} if for each $i=1,2,\dots$ 
\begin{enumerate}[label = (P5.\roman*)]
        \item \label{peripheral groups stabilize}$A_c = A_{f_i(c)}$ and
        \item \label{peripheral cosets stabilize}$g_e\ii g_{e'} A_c= g_{f_i(e)}\ii g^{}_{f_i(e')} A^{}_{f_i(c)}$ for $e,e'\in E\bfA$ satisfying $c=e_-=e'_-$.
\end{enumerate} 
Otherwise, $(\bbC,c)$ is \emph{uncontrolled}. 

\begin{step} Up to passing to a subsequence and changing the peripheral vertex and edge groups of $\calA$, we may further assume that 
 \begin{enumerate}[label = (P\arabic*)]\setcounter{enumi}{4}
     \item \label{peripherals stabilize} Each peripheral star $(\bbC, c)$ of $\calA$ is controlled.
 \end{enumerate}
\end{step}

\begin{claimproof}[Proof of Step 3] By Step 2 and passing to a subsequence, we may assume that trivial peripheral stars are controlled. We now consider non-trivial peripheral stars.


If there exists an uncontrolled peripheral star, then there exists an $\bbA$-path $\gamma$ beginning at the basepoint such that only its terminal vertex meets an uncontrolled peripheral star. 
For each $i=1,2,\dots$, there exists an $\bbA_i$-path $\gamma_i$ such that $\nu_\calA(\gamma)=\nu_{\calA_i}(\gamma_i)$ and $f_i$ carries the underlying path of $\gamma$ to the underlying path of $\gamma_i$.
Thus, up to conjugation by $\nu_\calA(\gamma)$, we can suppose the basepoint of $\bfA$ is a non-central vertex of an uncontrolled peripheral star $(\bbC,c)$ corresponding to a peripheral subgroup $P\in \bbP$.  
Let $e\in E\bfC$ be the edge such that $e_-=c$ and $e_+=v_0$ is the basepoint. 

Since peripheral subgroups are abelian, $A_{f_i(c)}=g_{f_i(e)}\ii A^{}_{f_i(c)} g^{}_{f_i(e)}$. Thus $A_{f_i(c)}$ is contained in a peripheral structure $\bbO_{\calA_i}$.
By \ref{WW: peripheral}, $\nu_{\calA_i}(\bbO_{\calA_i})$ is a $(G,\bbP)$ induced peripheral structure on $H_i$.
Since $A_{f_i(c)}\le P$, we must have $A_{f_i(c)} = H_i \cap P$. The ascending sequence $A_{f_1(c)}\le A_{f_2(c)}\le \dots$ stabilizes, since it is contained in $P$, which is finitely generated abelian. 
Hence, up to passing to a subsequence, we may assume that \ref{peripheral groups stabilize} holds for $(\bbC,c)$.

\medskip

We continue on to proving \ref{peripheral cosets stabilize}. For $i=1,2,\dots$, consider the set $$M_i = \{d_X(A_c,\;g_{f_i(e)}\ii g^{}_{f_i(e')} A^{}_c)\;|\;e'\in E\bfC,\; e'_-=c\}$$
where $e$ is the edge that was fixed before.
Note that by taking $e'=e$ we get that $0\in M_i$ for all $i$.
We will prove that there exists $R\ge 0$ such that for all $i\le j$ we have 
\begin{equation}\label{coset distances}
M_i \subseteq N_R(M_j).
\end{equation}
The number of elements in $M_i$ is bounded by the number of edges in $E\bfC$, which is constant by (P2). Thus \Cref{coset distances} implies \[\mu = \sup\cup_i M_i<\infty.\] 
For any $e',e''\in E\bbA$ satisfying $c=e'_-=e''_-$ we have
\begin{equation*}
    d_X(A_c, g_{f_i(e')}^{-1}g^{}_{f_i(e'')}A_c)  \le  d_X(A_c, g_{f_i(e)}^{-1}g^{}_{f_i(e')} A_c) +  d_X(A_c, g_{f_i(e)}^{-1}g^{}_{f_i(e'')}A_c)
    \le 2\mu
\end{equation*}
for all $i$.
There are finitely many translates $gA_c$ with $d_X(A_c,gA_c) \le 2\mu$. 
Thus, by passing to a further subsequence, we may assume \ref{peripheral cosets stabilize} holds for $(\bbC, c)$. 

It remains to show \Cref{coset distances} holds for some $R>0$. 
Fix $i\in \{1,2,\dots\}$, and let $$r:= d_X(A_c,\;g_{f_i(e)}\ii g^{}_{f_i(e')} A^{}_c)$$ for some $e,e'\in E\bfC$ with $e_-=e'_-=c$. 
If $r\le R$, then $r\in N_R(M_j)$ since $0\in M_j$, so we may
suppose $r\ge R$. 
Thus the problem is reduced to finding $R>0$ such that $r\ge R$ implies $r\in N_R(M_j)$.
By \Cref{rem: peripheral turns are realizable} and since $e_+=v_0$, there exists a reduced $\bbA_i$-path $t'_i$ of the form $(1,f_i(e\ii),1,f_i(e'),1,\dots)$ representing an element of $\pi_1(\bbA, v_0)$. 
Let $\bar s_i$ be a $\nu$-taut $C$-quasigeodesic realization of an $\bbA$-path $t_i\in [t'_i]$, by \ref{WW: paths and qi}.
Since both $t_i,t'_i$ are reduced, $t_i$ has the form $(a_0, f_i(e\ii), a_1, f_i(e'), a_2, \dots )$ for some $a_0,a_1,a_2\in A_c$. 
When forming $\bar s_i$, the maximal $f_i(\bbC)$-subpath of the form $(a_0, f_i(e\ii), a_1, f_i(e'), a_2)$ is replaced by an edge $\bar e_i$ with $\lab(\bar e_i)=p_i$ for some $p_i\in P$. 
Since $P$ is abelian, we have $r=d_X(A_c,p_iA_c)$.


Let $g = \nu_{\calA_i}([t_i])\in H_i$.
For any $j\ge i$, since $H_j\ge H_i$, there exists a realization $\bar s_j$ of a reduced $\bbA_j$-path $t_j$ such that $\bar s_j$ is a $\nu$-taut $C$-quasigeodesic by  \ref{WW: paths and qi} and $\nu_{\calA_j}([t_j])=g$. 
Let $\varepsilon = \varepsilon(C,\nu,G,\bbP,X)$ be the taut BCP constant.
By taking $R\ge 2\varepsilon$, we have $\len_X(\bar e_i)\ge r\ge 2\varepsilon$. 
Then, by \Cref{upgraded BCP} there exists a $P$-edge $\bar e_j$ in $\bar s_j$ connected to $\bar e_i$ such that
\begin{equation}\label{close endpoints of pi and pj}
d_X(\bar e_{i-},\bar e_{j-}),d_X(\bar e_{i+},\bar e_{j+})<\varepsilon.
\end{equation}

Letting $p_j = \lab(\bar e_j)$, we have 
$ p_j = u_1 p_i u_2$ where $|u_i|_X\le \varepsilon$, $u_i\in P$.
Since $P$ is abelian, we have $p_j=p_i u$ where $|u|_X\le 2\varepsilon$. In particular, 
\begin{align*}
    |p_j|_X &\ge d_X(A_c,p_jA_c) \\
    &\ge d_X(A_c,p_iA_c) - d_X(A_c,uA_c) \\
    & \ge r-2\varepsilon\ge R-2\varepsilon, \text{ so we have}\\
    \stepcounter{equation}|p_j|_X & \ge R-2\varepsilon.\tag{\theequation}\label{length of pj}
\end{align*}

Recall that $\bar s_j$ is obtained from $t_j$ by taking the path $s_j$ obtained by concatenating geodesics corresponding to the elements in $t_j$ and replacing each maximal peripheral star subpath by a single edge.
Thus, the edge $\bar e_j$ in $\bar s_j$ is obtained in one of the following ways:

\textbf{Case (i).} The edge $\bar e_j$ is an edge of a geodesic path corresponding to $g_y$ for some non-peripheral edge $y\in E\bfA_j$.

By Step 1, the length of such edges is uniformly bounded by some $B_1$. If $R > B_1+2\varepsilon$, then \eqref{length of pj} excludes this case.  

\textbf{Case (ii).} The edge $\bar e_j$ is an edge of a geodesic path $\alpha$ corresponding to an element $a\in A_v$ for some essential vertex $v\in V\bfA_j$.

Let the label of $\bar s_j$ be the concatenation $\sigma \alpha\sigma'$, and let the label of $\alpha$ be the concatenation $\tau p_j \tau'$.
Since $\bar e_j$ is connected to $\bar e_i$ then the prefix $\sigma \tau$ of the label of $\bar s_j$ is in $P$.
We have 
\begin{align*}
\label{Q is in Ac}
    Q &:= \tau \ii A_v \tau \cap P \\
    &=  (\sigma\tau)\ii (\sigma  A_v\sigma\ii) (\sigma  \tau) \cap P\\
    &= (\sigma  A_v\sigma\ii)\cap P\\ 
    &\le H_j\cap P=A_c
\end{align*}
where the third equality holds since $P$ is abelian.
By \Cref{peripheral coset distance in quasiconvex}, there exists $\rho(A_v)$ such that $d_X(p_j Q,Q)<\rho(A_v)$.
Since $Q\le A_c$, we get $$d_X(p_jA_c,A_c)\le d_X(p_jQ,Q)<\rho(A_v).$$
Let $\rho =\max_v\rho(A_v)$, where $v$ ranges over the finitely-many essential vertices in $V\bfA_j$.
Note that $\rho$ is independent of $j$, since the $\calA_j$ are all equivalent.
By \eqref{length of pj}, we may assume that we are not in this case by choosing $R\ge \rho+2\varepsilon$.

\textbf{Case (iii).} The edge $\bar e_j$ is the edge replacing a trivial peripheral star subpath.

By Step 2, the length of such edges is uniformly bounded by some $B_2$. 
Choosing $R>B_2+2\varepsilon$,  \eqref{length of pj} implies that we are not in this case.

\medskip
Taking $R=\max\{B_1,B_2,\rho\}+2\varepsilon$ rules out cases (i)-(iii), so we are left with the following. 
\medskip

\textbf{Case (iv).} The edge $\bar e_j$ is the edge replacing a non-trivial peripheral star subpath.

The label of $\bar s_j$ has the form $\sigma p_j \sigma'$ where $p_j$ is the label of $\bar e_j$. Let $(C',c')$ be the non-trivial peripheral star such that $\bar e_j$ replaces a $\bbC'$-subpath of $s_j$. Since $\bar e_j$ and $\bar e_i$ are connected, the peripheral subgroup associated to $(C', c')$ must be $P$ and consequently $\sigma\in P$.
Therefore,
$$\sigma A_{c'}\sigma^{-1}= \sigma P \sigma^{-1}\cap H_j = P\cap H_j = A_c.$$ 
By \ref{WW: peripheral}, it must be that $c=c'$. 
By \eqref{close endpoints of pi and pj}, letting $r'=d_X(A_cp_j,A_c)\in M_j$, we have, for $r\geq R$, 
$$|r-r'|= |d_X(A_c\; p_i ,A^{}_c)-d_X(A_cp_j,A_c)|\le 2\varepsilon\le R,$$
where the first inequality follows from $p_j=p_iu$ with $|u|_X\le 2\varepsilon$.
This establishes \eqref{coset distances}.

As explained above, it follows from \eqref{coset distances}, that up to passing to a subsequence, we may assume $(\calC, c)$ is controlled. 
The above procedure can be repeated until all peripheral stars are controlled.
\end{claimproof}

As a consequence of \ref{basepoint}-\ref{peripherals stabilize}, the images of $\nu_{\calA_i}$ for $i=1,2,\dots$ are equal. Thus, up to passing to a subsequence, $H_1\le H_2\le\dots$ is a constant sequence. This implies the original sequence stabilizes.
\end{proof}

\begin{corollary}\label{cor: ACC for hyperbolic 3-manifolds}
    Let $M$ be a finite-volume complete hyperbolic 3-manifold. Then $\pi_1(M)$ is $\omega$ACC.
\end{corollary}

\begin{proof}
    Let $H_1\le H_2\le\dots$ be an ascending sequence of bounded rank. 
    Since each $H_i$ is finitely generated it is
    a consequence of the Tameness Theorem~\cites{agol, calegari-gabai} and Canary's Covering Theorem~\cite{canary1996covering} for hyperbolic 3-manifolds that $H_i$ is either geometrically finite or a virtual fiber.

    If some $H_i$ is a virtual fiber, the chain stabilizes at $H_i$. This follows from the rank restriction:\footnote{When $M$ is closed all virtual fiber subgroups are closed surface groups. The following argument is only necessary in the cusped case.} Suppose $H_i$ is a virtual fiber and let $G_i$ be the finite-index subgroup of $G$ such that the corresponding cover of $M$ fibers over the circle with fiber subgroup $H_i$. Algebraically we have the short exact sequence 
    \[ 1\to H_i \to G_i\to \mathbb{Z}\to 1.\]
    Consider $H_j$ for $j\ge i$. The intersection $H_j \cap G_i$ is finite-index in $H_j$. Observe $H_i$ is normal in $H_j\cap G_i$, hence $[H_j\cap G_i : H_i] < \infty$. Thus $[H_j : H_i] < \infty$. However, $\rank(H_i) = \rank(H_j)$, so we conclude $H_i = H_j$.

    Otherwise, suppose $H_i$ are geometrically finite for all $i$.
    The group $\pi_1(M)$ is hyperbolic relative to the cusp groups. Geometrically finite translates to being relatively quasiconvex. Therefore, $H_i$ stabilizes by \Cref{main theorem rel Hyperbolic}.
\end{proof}



\section{Rank bounds for acylindrical graphs of groups}
\label{sec: rank bounds for acylindrical graphs of groups}

\begin{theorem}[Weidmann {\cite{weidmann2015rank}*{Theorem 0.1}}]\label{weidmann rank bound}
Let $\bbG$ be a finite, minimal, $k$-acylindrical graph of groups without trivial edge groups and with a finitely generated fundamental group. Then
\[
 \sum_{v \in V\bfG} \rank G_v - \sum_{e \in E\bfG} \rank G_e + |\calE\bfG| + b_1(\bfG) + 1 + 3k - \left\lfloor \frac{k}{2} \right\rfloor \le (2k+1)\rank\pi_1(\bbG).
\]
where $b_1(\bfG)$ is the first Betti number of $\bfG$.
\end{theorem}

Recall that $\calE\bfG$ is the set of geometric edges. In particular, $|\calE\bfG| = \tfrac12|E\bfG|.$

\begin{corollary}\label{simplified weidemann rank bound}
    Let $\bbG$ be a finite, $k$-acylindrical graph of groups with a finitely generated fundamental group. Then
\begin{equation}\label{simplified rank bound}
 \sum_{v \in V\bfG} \rank G_v - \sum_{e \in E\bfG} \rank G_e \le (2k+1)\rank\pi_1(\bbG) .
\end{equation}
\end{corollary}

\begin{definition}
If $\bfA$ is a connected subgraph of $\bfG$, we let $\bbG\vert_{\bfA}$ denote the restriction of $\bbG$ to $\bfA$.
\end{definition}

\begin{proof}
    \textbf{Case 1:} If $\bbG$ is minimal without trivial edge groups then the corollary follows directly from \Cref{weidmann rank bound}. 

    \textbf{Case 2:} If $\bbG$ has no trivial edge group but is not minimal, then the minimal subgraph of groups is obtained by a finite sequence of leaf removals (see \Cref{minimality criterion}). More precisely, if $\bbG$ is not minimal, then there exists some vertex $e'_+$ of degree 1, such that $G_{e'_+}=G_{e'}$.
    By removing $e'_+$ and the geometric edge $\{e',(e')\ii\}$ one obtains a smaller graph of groups $\bbG-e'$ with underlying graph $\bfG-e'$ whose fundamental group is the same as that of $\bbG$.
    Note that $\rank G_{e'_+}=\rank G_{e'}$ and so $$\sum_{v \in V(\bfG-e')} \rank G_v - \sum_{e \in E(\bfG-e')} \rank G_e \ge \sum_{v \in V\bfG} \rank G_v - \sum_{e \in E\bfG} \rank G_e$$ 
    Thus inequality \eqref{simplified rank bound} holds for $\bbG$ if it holds for $\bbG-e'$. 
    
    Since this process eventually terminates with the minimal subgraph $\bbG'$ of groups of $\bbG$. By Case 1, \eqref{simplified rank bound} holds for $\bbG'$ and so it holds throughout the process, and in particular holds for $\bbG$.
    
    \textbf{Case 3:} If $\bbG$ has some trivial edge groups, then let $\bfA_1,\dots,\bfA_k$ denote the components of the complement of the trivial edges in $\bfG$.
    Let $A_i = \pi_1(\bbG\vert_{\bfA_i}) \le G$. 
    Then, $G = A_1 * \dots *A_k *F_r$ for some $0\le r\in \bbZ$, and so by Grushko's Theorem
    \begin{equation}\label{eq: grushko} 
    \sum_{i=1}^k\rank A_i \le \rank G.
    \end{equation}
    By Case 2, \eqref{simplified rank bound} holds for each $A_i=\pi_1( \bbG|_{\bfA_i})$. That is,
    $$ \sum_{v\in V\bfA_i} \rank G_{v} - \sum_{e\in E\bfA_i} \rank G_e \le (2k+1)\rank  A_i.$$
    Summing over all $i=1,\dots,k$ and using \eqref{eq: grushko} gives \eqref{simplified rank bound} for $G$.
\end{proof}

The acylindricity assumption in \Cref{simplified weidemann rank bound} is necessary (see \cite{weidmann2002nielsen}*{Theorem 5}), and the statement can be viewed as a partial converse to the following observation.

\begin{observation}\label{rank observation}
    Let $\bbG$ be a finite graph of groups, and let $G = \pi_1(\bbG)$, then $$\rank(G) \le \sum_{v\in V\bfG} \rank G_v  + b_1(\bfG)$$
\end{observation}





    
    
        
\section{Complexity of subgroups}
\label{sec: combination theorem}

\begin{notation}
    Let $G = \pi_1(\bbG)$. For a finitely generated $H\le G$, let $T_H$ denote a minimal tree for $H\actson T_\bbG$. Such a tree exists if $H$ is finitely generated, and it is unique if $H\actson T_\bbG$ is non-elliptic. Let $\bbG_H$ denote the graph of groups for the action $H\actson T_H$. 
\end{notation}

Note that $\bbG_H$ is finite since $H$ is finitely generated.
If $H_1\le H_2$ are non-elliptic finitely generated subgroups of $G$, then the inclusion $T_{H_1}\hookrightarrow T_{H_2}$ descends to a graph morphism $\bfG_{H_1}\to \bfG_{H_2}$.

\begin{definition}\label{definition of complexity}
    For a finitely generated $H\le G$, the \emph{complexity of $H$} is $$C(H) = (\varepsilon'(\bbG_H) , \rank H)$$ with the lexicographic order, where $\varepsilon'(\bbG_H)$ is the number of geometric edges in $\bfG_H$ with a non-trivial edge group.
\end{definition}

\begin{setup}\label{setup}
    Let $\bbG$ be a $k$-acylindrical, finite graph of groups, such that all edge groups $G_e$ are finitely generated abelian groups.
\end{setup}

\begin{lemma}\label{complexity bounds rank}
    Assume \Cref{setup}. Let $\bbG_H$ be the graph of groups associated to a finitely generated subgroup $H\le G$. Let $\bfK\subset \bfG_H$ be a proper connected subgraph  and let $K=\pi_1(\bbG|_{\bfK})$. Then:
    \begin{enumerate}
        \item $C(K)<C(H)$.
        \item There exists $\gamma=\gamma(\bbG,m,r)$ such that: If $\rank H\le r$ and $\varepsilon'(\bbG_H)\le m$, then $\rank K\le \gamma$.
    \end{enumerate}
\end{lemma}

\begin{proof}
    The graph of groups $\bbG_K$ can be obtained from $\bbG_H$ by deleting a sequence of edges, first to obtain $\bbG|_\bfK$ and then to make it minimal as in \Cref{minimality criterion}. 
    In particular, $\bfG_K$ is a subgraph of $\bfK$. 
    No edge removal increases $\varepsilon'$. Since $\bbG_H$ is minimal, removing the first edge $e$ either: 
    \begin{itemize}
        \item strictly decreases $\varepsilon'$ when $G_e\ne 1$, or 
        \item strictly decreases $\rank$ when $G_e = 1$.
    \end{itemize} The inequality $C(K)<C(H)$ follows.

    \medskip
    
    Let $p=p(\bbG)$ be the maximal rank of an edge group of $\bbG$. Since edge groups are abelian, $p$ bounds the rank of the edge subgroups of $\bbG_H$. Then we have $$\sum_{e\in E\bfG_H}\rank G_e\le pm.$$ 
    Applying \Cref{simplified weidemann rank bound} to $\bbG_H$ and combining with the above inequality, we have $$\sum_{v\in V\bfG_H}\rank G_v\le (2k+1)\rank H +pm \le (2k+1)r+pm.$$
    Since $\pi_1(\bfG_H)$ is a retract of $H$, we have $b_1(\bfG_H)\le\rank H \le r$. By \Cref{rank observation}, we may bound $\rank K$ as follows
    \begin{align*}
        \rank K &\le \sum_{v\in V\bfG_K}\rank G_v + b_1(\bfG_K)\\
        &\le \sum_{v\in V\bfG_H}\rank G_v + b_1(\bfG_H)\\
        &\le (2k+2)r+pm.\qedhere
    \end{align*}
\end{proof}

\begin{lemma} \label{surjective cores or elliptic} 
Assume \Cref{setup}. Let $H_1\le H_2\le \dots$ be an ascending chain of subgroups of $G$ such that $\sup_i\{\varepsilon'(\bbG_{H_i}), \rank H_i\}<\infty$. There exists an ascending chain $K_1\le K_2\le\dots$ of uniformly bounded rank subgroups of $G$ such that the following hold: 
    \begin{itemize}
        \item Either all $K_i$ are elliptic, or all $K_i$ are non-elliptic and each $\bfG_{K_i}\to \bfG_{K_{i+1}}$ is surjective.
        \item If $K_1\le K_2\le \dots$ stabilizes, then $H_1\le H_2\le\dots$ stabilizes.
    \end{itemize}
\end{lemma}

\begin{proof}
    By induction on $(m,q)$ with lexicographic order, we prove that if the sequence $H_1\le H_2\le \dots$ satisfies $C(H_i)\le (m,q)$ for all $i\in \bbN$ and $\sup \rank H_i <\infty$ then the lemma holds for the sequence $H_i$. The base case, $C(H_i)=(0,0)$ for all $i$, implies each $H_i$ is trivial. Letting $m=\sup\varepsilon'(\bbG_{H_i})$ suppose there exists $q\geq 0$ such that $C(H_i) \le (m,q)$ for all $i$.

    If infinitely many of the $H_i$ are elliptic, then all the $H_i$ are elliptic, and the lemma holds for $K_i=H_i$.

    Thus, up to removing finitely many of the $H_i$, we may suppose each $H_i$ acts non-elliptically and hence has a minimal tree $T_{H_i}$. If only finitely many of the maps $\bfG_{H_i}\to \bfG_{H_{i+1}}$ are non-surjective, then again after removing finitely many subgroups we may assume each $\bfG_{H_i}\to \bfG_{H_{i+1}}$ is surjective. The conclusion holds for $K_i=H_i$. 

    Thus, passing to a subsequence, we may assume each $H_i$ acts non-elliptically and for infinitely many indices $i$, the map $\bfG_{H_i}\to \bfG_{H_{i+1}}$ is non-surjective. For each such $i$, let $K_{i+1}=\pi_1(\bbG|_{\bfA_i})$ where $\bfA_i$ is the image of $\bfG_{H_{i-1}}\to \bfG_{H_i}$.
    Since $H_i \le K_{i+1} \le H_{i+1}$, the $K_i$ sequence is ascending and stabilizes if and only if the $H_i$ sequence stabilizes. By \Cref{complexity bounds rank}, $C(K_i)<C(H_i)$ and $\rank(K_i)\leq \gamma(\bbG,m,r)$ where $r=\sup\rank H_i$. Additionally, we have:
    \begin{itemize}
        \item $C(K_i)\leq (m,q-1)< (m,q)$ if $q>1$.
        \item $C(K_i)\leq (m-1,\gamma(\bbG,m,r))< (m,0)$ if $q=0$.
    \end{itemize}
   Hence the lemma follows by induction.
\end{proof}

\begin{lemma}\label{bounded complexity implies ACC}
    Assume \Cref{setup}. Suppose that each vertex group $\{G_v\}_{v\in V\bfG}$ satisfies $\omega$ACC. Then any ascending chain of subgroups with uniformly bounded rank and complexity stabilizes.
\end{lemma}

\begin{proof}
    Let $H_1\le H_2\le \dots$ be an ascending chain of subgroups with uniformly bounded rank and complexity, and let $K_1\le K_2\le\dots$ be the chain of subgroups obtained by applying \Cref{surjective cores or elliptic} to $H_1\le H_2\le\dots$. 
    
    Suppose first that each $K_i$ is elliptic. Since the $K_i$ form an ascending chain, the fixed sub-tree of $K_{i+1}$ is contained in the fixed sub-tree of $K_i$. By acylindricity of the action, these sub-trees have bounded diameter, so their intersection is non-empty. In particular, there exists a vertex group $G_v$ containing each $K_i$. The lemma then follows from the $\omega$ACC property of $G_v$. 

    Suppose now that each $K_i$ is non-elliptic and each $\bfG_{K_i}\to \bfG_{K_{i+1}}$ is surjective. 
    Passing to a subsequence, we may suppose that each $\bfG_{K_i}\to \bfG_{K_{i+1}}$ is a graph isomorphism. Let $\bfK$ be the graph isomorphic to each $\bfG_{K_i}$ for all $i$. Since $K_i\le K_{i+1}$, the minimal tree $T_{K_i}$ is contained in $T_{K_{i+1}}$. 
    Moreover, since $\bfG_{K_i}\to \bfG_{K_{i+1}}$ is an isomorphism, the inclusion $T_{K_i}\to T_{K_{i+1}}$ induces bijections between the vertex/edge orbits of $T_{K_i}$ and the vertex/edge orbits of $T_{K_{i+1}}$. 
    Note we are considering the $K_i$-orbits within $T_{K_i}$ and the $K_{i+1}$-orbits within $T_{K_{i+1}}$. 
    Choosing appropriate orbit representatives, we have $G_{v,1}\leq G_{v,2}\le\dots$ for each vertex $v\in V\bfK$, and $G_{e,1}\le G_{e,2}\le\dots$ for each edge $e\in E\bfK$. 
    Since the vertex and edge groups of $G$ satisfy $\omega$ACC, the graphs of groups eventually stabilize. 
    Similarly, the maps $G_e\hookrightarrow G_{e_+}$ stabilize. Hence the inclusions $K_i\hookrightarrow K_{i+1}$ are eventually isomorphisms. 
\end{proof}

\subsection{Freely indecomposable factors}

\begin{theorem}[Weidmann {\cite{weidmann2002nielsen}*{Theorem 1}}]\label{weidmann indecomposable vertex bound}
Let $G$ be a non-cyclic, freely indecomposable, finitely generated group, and let $G \actson T$ be a minimal $k$-acylindrical action. Then there are at most $2k (\rank G - 1)+1$ vertices in $T/G$.
\end{theorem}

\begin{corollary}\label{weidmann indecomposable complexity bound}
    Let $G$ and $G\actson T$ satisfy the hypotheses of \Cref{weidmann indecomposable vertex bound}, and let $\bfG=T/G$. The number of edges of $\bfG$ is bounded by a function of $\rank(G)$ and $k$. Consequently, the complexity $C(G)$ is bounded by a function of $\rank(G)$ and $k$. 
\end{corollary}

\begin{proof}
    We have the equality $|E\bfG|=|V(G)|+b_1(\bfG)-1$ which holds for any connected graph. By \Cref{weidmann indecomposable vertex bound}, $|V\bfG|\leq 2k(\rank(G)-1)+1$. We also have the trivial bound $b_1(\bfG)\leq \rank(G)$. Substituting the inequalities into the equation, we get $$\varepsilon'(\bbG)\le |E\bfG|\leq (2k+1)\rank(G)-2k.$$ It follows $C(G)$ is also bounded by a function of $\rank(G)$ and $k$.
\end{proof} 

    \begin{lemma}\label{reduce rank by conjugates}
        Let  $B_1,\dots,B_n$  be subgroups of a group $H$ and let $A_1,\dots,A_t$ be representatives of the conjugacy classes of $B_1,\dots,B_n$ in $H$. There exists $K$ such that $\gen{B_1,\dots,B_n} \le K \le H$ and $$\rank(K) \le \sum_{j=1}^t\rank(A_j) + n - t.$$
    \end{lemma}
    \begin{proof}
        For $1\le i \le t$, let $B_{i,1},\dots,B_{i,n_i}$ be the subgroups among $B_1,\dots,B_n$ that are conjugate to $A_i$ in $H$. They are pairwise conjugate and so there exist $h_{i,2},\dots,h_{i,n_i}$ such that $h_{i,j}B_{i,1}h_{i,j}\ii = B_{i,j}$. Set
        $K_i:=\gen{B_{i,1},h_{i,2},\dots,h_{i,n_i}}$.
        Then, $\gen{B_{i,1},\dots,B_{i,n_i}} \le K_i \le H$ and $$\rank(K_i) \le \rank(B_{i,1}) + n_i-1 = \rank(A_i)+n_i-1.$$

        Finally, set $K = \gen{K_1,\dots,K_t}$. Then, $\gen{B_1,\dots,B_n} \le \gen{K_1,\dots,K_t} = K \le H$ and 
        \begin{align*}
            \rank(K) &\le \sum_{i=1}^t \rank(K_i) \\
            &\le  \sum_{i=1}^t (\rank (A_i) +n_i -1) \\
            &= \sum_{i=1}^t\rank (A_i) + n - t. \qedhere
        \end{align*}
    \end{proof}

\begin{lemma}\label{from bounded rank to isomorphic}
    Assume the setting of \Cref{setup}. Additionally, suppose
    \begin{itemize}
        \item Torsion elements of $G$ have uniformly bounded order.
        \item The vertex groups $\{G_v\}_{v\in V\bfG}$ each satisfies $\omega$ACC.
    \end{itemize} 
    If there exists a proper ascending chain of uniformly bounded rank subgroups of $G$, then there exists a proper ascending chain of isomorphic (finitely generated) subgroups of $G$. 
\end{lemma}

\begin{proof}
    Let $r\in \bbN$ be the minimal number for which there exists a proper ascending chain $H_1< H_2< \dots $ of subgroups of $G$ with $\rank(H_i)\le r$.
    
    Using the Grushko decomposition, write each $H_i$ as a free product of subgroups $H_i = A^1_i*\dots *A^{n_i}_i * C^1_i*\dots *C^{k_i}_i$ where the $A^j_i$ are freely indecomposable non-cyclic groups, and the $C^j_i$ are cyclic. 
    Since the $H_i$ have uniformly bounded rank, there exist $n,k\ge 0$ so that, after passing to a subsequence, $n_i=n$, $k_i=k$ for all $i$. Moreover, since $G$ has uniformly bounded torsion, there exist cyclic groups $C^1,\dots, C^k$ so that $C^j_i \simeq C^j$ for all $1\le j \le k$. 
    
    For each $i\geq 1$, there exists a function $\sigma_i : \{1,\dots,n\}\to \{1,\dots,n\}$ such that for all $1\le j \le n$ the factor $A_i^j$ is a subgroup of a unique conjugate $B_i^j$ of the factor $A_{i+1}^{\sigma_i(j)}$ in $H_{i+1}$. 
    
    \begin{claim} $\sigma_i$ is bijective for all but finitely many $i$. 
    \end{claim}

    \begin{claimproof}[Proof of Claim]
    Each $B^j_i$ is a conjugate of $A^{\sigma_i(j)}_{i+1}$ in $H_{i+1}$. Therefore, by \Cref{reduce rank by conjugates}, there exists some $K'_i$ such that $\gen{B^1_i,\dots,B^n_i} \le K'_i \le H_{i+1}$ and 
    $$
    \rank(K'_i) \le \sum_{j\in \im(\sigma_i)}\rank(A^j_{i+1}) + n - |\im(\sigma_i)|
    $$
    
    If $\sigma_i$ is not a bijection, then $n\ne |\im(\sigma_i)|$. Since $\rank(A^j_{i+1})>1$ for all $j$, we have
    $$n - |\im(\sigma_i)| < \sum_{j\not\in \im\sigma}\rank(A^j_{i+1}).$$
    
    Combining this with the above gives
    \begin{align*}
        \rank(K_i')&< \sum_{j\in \im(\sigma_i)}\rank(A^j_{i+1}) + \sum_{j\not\in \im\sigma}\rank(A^j_{i+1})  \\
        & = \sum_{j}\rank(A^j_{i+1})
    \end{align*}

    Define $K_i = \gen{K_i',C^1_i,\dots,C^k_i}$.
    Then, $H_i \le K_i \le H_{i+1}$ and 
    $$\rank(K_i) \le \rank(K_i') + k < \sum_{j}\rank(A^j_{i+1})+ k = \rank(H_{i+1})=r. $$

     If $\sigma_i$ is not a bijection for infinitely many $i$, the corresponding $K_i$ subsequence is a proper ascending chain of smaller rank, contradicting our assumption on $r$.
     This completes the proof of the claim.
     \end{claimproof}

    Now, since the $\sigma_i$ are eventually bijective, up to reordering the $A_j^i$ and passing to a subsequence, we may assume $\sigma_i = \id$ for all $i$. The ascending sequence $H_1\le H_2 \le \dots$ gives rise to ascending sequences of the non-cyclic factors. That is, for every $1\le j \le n$ there exists a sequence $\bar A^j_1 \le \bar A_2^j\le \dots  $ where $\bar A^j_i$ is a conjugate of $A^j_i$.
    The $\bar A^j_i$ are freely indecomposable non-cyclic groups of uniformly bounded rank, so by \Cref{weidmann indecomposable complexity bound}, they have uniformly bounded complexity. By \Cref{bounded complexity implies ACC} the sequence $\bar A^j_1 \le \bar A_2^j\le \dots  $ stabilizes. Thus, for each $1\le j \le n$, we eventually have $A^j_i \simeq A^j$ for some group $A^j$.

    We have proven that, up to passing to a subsequence, all the subgroups $H_i$ are isomorphic to the same group $A^1 * \dots *A^n*C^1 *\dots *C^k$.
\end{proof}

\subsection{Chains of isomorphic subgroups}

\begin{theorem}\label{nontrivial edge bound}
    Let $G$ be a finitely presented group, and let $k\in \bbN$. There exists $\eta = \eta(G,k)$ such that for every $k$-acylindrical minimal action of $G$ on a tree $T$ we have $\varepsilon'(T/G) \le \eta$.
\end{theorem}

\begin{definition}[Reduced]
    Let $\bbG$ be a finite graph of groups. An edge $e\in E\bfG$ with distinct endpoints is \emph{reducible} if $e_+$ has valence two and $G_e\hookrightarrow G_{e_+}$ is an isomorphism. 
    A minimal graph of groups is \emph{reduced} if there are no reducible edges, and an action on a tree $G\actson T$ is \emph{reduced} if the graph of groups structure on $T/G$ is reduced.
\end{definition}

\begin{theorem}[Bestvina--Feighn { \cite{bestvinaFeign1991boundingComplexity}*{Main Theorem}}]\label{bestvina feign}
    Let $G$ be a finitely presented group. There exists $\gamma=\gamma(G)$ such that every reduced $G$-tree with small edge stabilizers has at most $\gamma$ orbits of vertices.
\end{theorem}

\begin{definition}[Reduction]\label{reduction map}
    If $e$ in $\bbG$ is reducible, then we can view $G_{e_+}$ as a subgroup of $G_{e_-}$.  The graph of groups $\bbG/e$ is obtained by contracting the geometric edge $\{e, e \ii \}$ to a point (using $G_{e_+}\le G_{e_-}$).
    We say that $\bbG/e$ is obtained from $\bbG$ by a \emph{reduction along} $e$. 
\end{definition}

\begin{proof}[Proof of \Cref{nontrivial edge bound}]
    Let $G\actson T$ be a minimal $k$-acylindrical action with associated graph of groups $\bbG$. 
    One can perform a sequence of reductions
    $$\bbG \to \bbG/e_1 \to \bbG/(e_1\cup e_2) \to \dots \to \bbG/(e_1\cup \dots \cup e_n) = \bbG''$$
    to obtain a reduced graph of groups $\bbG''$.
    Moreover, we can arrange $e_1,\ldots, e_m$ to have trivial edge group and $e_{m+1},\ldots, e_n$ to  have nontrivial edge groups for some $m$. 
    We denote $\bbG'  = \bbG/(e_1\cup \dots \cup e_m)$, so we have $\bbG \to \bbG' \to \bbG''$.
    Let $V,\calE$ (resp. $V',\calE'$ and $V'',\calE''$) be the vertices and geometric edges of $\bfG$ (resp. $\bfG'$ and $\bfG''$). 
    By \Cref{bestvina feign}, 
    \begin{equation*}
        |V''|\le \gamma(G).
    \end{equation*}
    Thus we have,
    \begin{align*}\label{eq: bound of E''}
        |\calE'' |&= |\calE''| - |V''| + |V''|\\
        &= -\chi(\bfG'') + |V''|\\
        &=b_1(\bfG'')-1+|V''|\\
        &\le \rank(G)-1+\gamma(G).
    \end{align*}

    \begin{claim}\label{preimage of reduction}
        Let $r$ be the quotient $\bbG'\to \bbG''$. For each vertex $v\in \bbG''$, the preimage $r\ii(v)$ is a tree with at most $k\deg(v)$ geometric edges.
    \end{claim}

    \begin{claimproof}[Proof of Claim]
        Consider the sequence of reductions $$\bbG'\to \bbG'/e_{m+1}\to \bbG'/(e_{m+1}\cup e_{m+2})\to\dots\to \bbG'/(e_{m+1}\cup\dots \cup e_n)=\bbG''.$$ Because a reducible edge must have an endpoint of degree two, each quotient map $\bbG'/(e_{m+1}\cup\dots\cup e_i)\to \bbG'/(e_{m+1}\cup\dots\cup e_{i+1})$ preserves degrees of vertices. Thus the inverse image in $\bbG'$ of a vertex in any $\bbG'/(e_{m+1}\cup\dots\cup e_i)$ is a wedge-sum of paths. By $k$-acylindricity, each path in such a wedge-sum can have length at most $k$. In particular, the inverse image $r^{-1}(v)$ is a tree with at most $k\deg(v)$ geometric edges.
    \end{claimproof}
    
    By \Cref{preimage of reduction}, the preimage $r^{-1}(v)$ of a vertex $v\in \bbG''$ under $r:\bbG'\to \bbG''$ contains at most $k\deg(v)$ geometric edges. Thus we have
    $$|\calE'| - |\calE''| \le \sum_{v\in V''} k \deg(v)=2k|\calE''|.$$
    Since no edge with nontrivial stabilizer is reduced under the quotient $\bbG\to \bbG'$, we have
    $$ \varepsilon '(\bbG)\le |\calE'| \le (2k+1)|\calE''| \le (2k+1)(\rank(G)-1+\gamma(G)). $$
    Set $\eta(G,k)=  (2k+1)(\rank(G)-1+\gamma(G))$.
\end{proof}

\section{Three manifold groups}
\label{sec: omegaACC for three manifold groups}


We are now ready to establish our main combination theorem.

\begin{reptheoremA}{main combination thm}
    Let $G$ be the fundamental group of a $k$-acylindrical, finite graph of groups. Suppose:
    \begin{itemize}
        \item The edge groups are finitely generated abelian groups.
        \item The vertex groups satisfy $\omega$ACC.
        \item $G$ is coherent.
        \item Torsion elements of $G$ have uniformly bounded order.
    \end{itemize} 
    Then $G$ satisfies $\omega$ACC.
\end{reptheoremA}

\begin{proof}
    Let $G$ be as in the theorem, and let $H_1\le H_2 \le \dots$ be an ascending sequence of bounded rank subgroups of $G$. 
    By \Cref{from bounded rank to isomorphic}, we may assume that the $H_i$ are all isomorphic to some group $H$.
    By coherence $H$ is finitely presented, and so by \Cref{nontrivial edge bound} $\varepsilon'(H_i)\le \gamma(H,k)$.
    It follows that the $H_i$ have bounded complexity.
    Hence, by \Cref{bounded complexity implies ACC} the sequence stabilizes.
\end{proof}

Our main application of \Cref{main combination thm} is \Cref{main theorem 3-manifolds}.
We have already proved the $\omega$ACC property for the hyperbolic pieces in the geometrization decomposition of $M$ (\Cref{cor: ACC for hyperbolic 3-manifolds}). The next two lemmas establish $\omega$ACC for Seifert-fibered and Sol 3-manifolds.

\begin{lemma}\label{extension lemma}
Let $1\to N \to G \to S \to 1$ be a short exact sequence of groups. If one of the following holds:
\begin{enumerate}
    \item $N$ is Noetherian and $S$ satisfies $\omega$ACC, or
    \item $N$ is $\omega$ACC and $S$ is finite,
\end{enumerate} 
then $G$ satisfies $\omega$ACC.
\end{lemma}

\begin{proof}
Let $H_1\le H_2 \le \ldots \le G$ be an ascending chain of constant rank subgroups of $G$. The chain $H_i \cap N$ stabilizes in each case because:
\begin{enumerate}
    \item $N$ is Noetherian, or
    \item $\rank(H_i\cap N) \le \rank (H_i)\cdot |S|$ and $N$ is $\omega$ACC.
\end{enumerate}
The image chain $\bar{H}_i \le S$ stabilizes in both cases. Hence, the short five lemma implies $H_i$ stabilizes. 
\end{proof}

\begin{corollary}\label{acc for finite index}
    Let $G'\le G$ be a finite index subgroup. Then $G$ is $\omega$ACC if and only if $G'$ is $\omega$ACC.
\end{corollary}

\begin{proof}
    The forward direction is obvious. 
    For the reverse direction, we pass to the normal core of $G'$ and apply \Cref{extension lemma}.
\end{proof}

\begin{corollary}\label{no ascending in SFS}
Let $M$ be a Seifert-fibered or Sol 3-manifold. Every ascending chain constant rank subgroups in $\pi_1(M)$ stabilizes. 
\end{corollary}

\begin{proof}
Let $M$ be a Seifert-fibered 3-manifold, and let $G=\pi_1(M)$. Then, $G$ fits into an exact sequence
\[ 1\to N \to G \to S \to 1\] where $N\simeq \bbZ$ is the fiber group and $S$ is the fundamental group of a good 2--orbifold. 
The group $S$ is virtually a surface group.  By \cite{shusterman}, surface groups are $\omega$ACC, and therefore so is $S$ by \Cref{acc for finite index}.
Since $N\simeq \bbZ$ is Noetherian, The corollary follows from \Cref{extension lemma}.

A similar proof works for Sol 3-manifolds as they are virtually $\bbZ^2$-by-$\bbZ$ groups.
\end{proof}

We now assemble the pieces to prove the 3-manifold theorem.

\begin{reptheoremA}{main theorem 3-manifolds}
    Fundamental groups of compact 3-manifolds satisfy $\omega$ACC.
\end{reptheoremA}

\begin{proof}
First, we may assume that $M$ is closed by considering its double (note that $\pi_1(M)$ is a subgroup of $\pi_1(M\sqcup_{\partial M} M)$).
Suppose $M$ is a closed 3-manifold. 
By \Cref{acc for finite index}, we may pass to a finite-sheeted cover, so assume $M$ is orientable. By the prime decomposition~\cite{hempel}*{Theorem 3.15}, JSJ decomposition~\cites{jaco-shalen, johannson}, and geometrization~\cites{perelman1,perelman2}\footnote{See also Bessi\'eres et al.~\cite{bessieres}, Kleiner and Lott~\cite{kleiner-lott}, and Morgan and Tian~\cite{morgan-tian}.} theorems, $M$ decomposes along essential spheres and incompressible tori into finite-volume hyperbolic and Seifert-fibered pieces. 
This decomposition splits $\pi_1(M)$ as an acylindrical graph of groups with abelian edge groups \cite{lafont}*{Proposition 8.2}.
By \Cref{cor: ACC for hyperbolic 3-manifolds,no ascending in SFS} the vertex groups of this splitting satisfy $\omega$ACC. 
By Scott \cites{scott1973coherence, scott1973compact}, fundamental groups of 3-manifolds are coherent.
Torsion elements in $\pi_1(M)$ have uniformly bounded order: 
For example, a torsion element must fix a vertex in the JSJ decomposition, and the torsion in hyperbolic and Seifert-fibered manifolds is bounded.
Therefore, by \Cref{main combination thm}, $\pi_1(M)$ satisfies $\omega$ACC.
\end{proof}

\bibliography{biblio}
\bibliographystyle{plain}

\end{document}